\documentclass[11pt]{amsart}

\usepackage{amssymb, textcomp, dsfont}
\usepackage[all]{xy}
\usepackage{hyperref}
\usepackage{aliascnt}
\usepackage{color}
\usepackage{soul}
\usepackage{mathtools,array}
\usepackage{tikz-cd}
\usepackage{lmodern}
\usepackage[inline]{enumitem}
\usepackage{xcolor}

\numberwithin{equation}{section}

\newtheorem{lma}{Lemma}[section]

\newaliascnt{thmCt}{lma}
\newtheorem{thm}[thmCt]{Theorem}
\aliascntresetthe{thmCt}

\newaliascnt{corCt}{lma}
\newtheorem{cor}[corCt]{Corollary}
\aliascntresetthe{corCt}

\newaliascnt{prpCt}{lma}
\newtheorem{prp}[prpCt]{Proposition}
\aliascntresetthe{prpCt}

\newtheorem*{thm*}{Theorem}
\newtheorem*{cor*}{Corollary}
\newtheorem*{prp*}{Proposition}

\theoremstyle{definition}

\newaliascnt{pgrCt}{lma}
\newtheorem{pgr}[pgrCt]{}
\aliascntresetthe{pgrCt}

\newaliascnt{dfnCt}{lma}

\aliascntresetthe{dfnCt}

\newaliascnt{rmkCt}{lma}
\newtheorem{rmk}[rmkCt]{Remark}
\aliascntresetthe{rmkCt}

\newaliascnt{rmksCt}{lma}

\aliascntresetthe{rmksCt}

\newaliascnt{prblCt}{lma}
\newtheorem{prbl}[prblCt]{Problem}
\aliascntresetthe{prblCt}

\newaliascnt{exaCt}{lma}

\aliascntresetthe{exaCt}

\newaliascnt{exasCt}{lma}

\aliascntresetthe{exasCt}

\newaliascnt{conjCt}{lma}

\aliascntresetthe{conjCt}

\newcommand{\N}{\mathbb{N}}

\newcommand{\K}{\mathrm{K}}

\DeclareMathOperator{\Cu}{Cu}

\makeatletter

\newcommand{\Rmnum}[1]{\expandafter\@slowromancap\romannumeral #1@}
\makeatother

\SetLabelAlign{Center}{\makebox[2em]{#1}}
\newcommand{\beginEnumStatements}{\begin{enumerate}[label=(\arabic*),align=Center,leftmargin=*, widest=iiii]}
\newcommand{\beginEnumConditions}{\begin{enumerate}[label={\rm (\roman*)},align=Center, leftmargin=*, widest=iiii]}

\usepackage{stmaryrd}

\newcommand{\cc}{\subset\!\subset}
\newcommand{\tiphi}{\tilde{\phi}}
\newcommand{\tipsi}{\tilde{\psi}}
\newcommand{\tialpha}{\tilde{\alpha}}

\newcommand{\V}{\mathrm{V}}
\newcommand{\CP}{\mathrm{CP}}
\newcommand{\Id}{\mathrm{Id}}

\theoremstyle{plain}
\newcounter{thmintroctr}

\newtheorem{thmintro}[thmintroctr]{Theorem}
\newtheorem{corintro}[thmintroctr]{Corollary}

\theoremstyle{definition}
\newcounter{goalintroctr}

\begin{document}

\title[Cuntz semigroups and stable rank one]{The Cuntz semigroup of rings with stable rank one}

\author{Pere Ara}
\author{Francesc Perera}

\author{Guillem Quingles}

\date{\today}

\address{P.~Ara,
Departament de Matem\`{a}tiques,
Universitat Aut\`{o}noma de Barcelona,
\linebreak 08193 Bellaterra, Barcelona, Spain}
\email[]{Pere.Ara@uab.cat}
\address{
F.~Perera, 
Departament de Matem\`{a}tiques,
Universitat Aut\`{o}noma de Barcelona,
\linebreak 08193 Bellaterra, Barcelona, Spain, and
Centre de Recerca Matem\`atica, Edifici Cc, Campus de Bellaterra,  08193 Cerdanyola del Vall\`es, Barcelona, Spain}
\email[]{Francesc.Perera@uab.cat}
\urladdr{https://mat.uab.cat/web/perera}

\address{Guillem~Quingles, 
Departament de Matem\`{a}tiques,
Universitat Aut\`{o}noma de Barcelona,
\linebreak 08193 Bellaterra, Barcelona, Spain}
\email{Guillem.Quingles@uab.cat}



\thanks{
 The authors were partially supported by the Spanish State Research Agency (grant No.\  PID2023-147110NB-I00), by the Comissionat per a Universitats i Recerca de la Generalitat de Ca\-ta\-lu\-nya (grant No.\ 2021-SGR-01015) and by the Spanish State Research Agency through the Severo Ochoa and María de Maeztu Program for Centers and Units of Excellence in R\&D (CEX2020-001084-M).
 The third named author was also partially supported by the grant PRE2021-099580.
 }

\begin{abstract}
We study the Cuntz semigroup of arbitrary rings with stable rank one. We prove that Cuntz subequivalence between countably generated projective right modules $P$, $Q$ is equivalent to $P$ being isomorphic to a pure submodule of $Q$, whilst Cuntz equivalence amounts to isomorphism.~We also prove order-cancellation of finitely generated projective modules (and, more generally, pairs of modules admitting certain complements) from direct sums with countably generated projective modules.
\end{abstract}

\maketitle




\section{Introduction}
The category of projective right modules on a
(unital) ring $R$ encodes a great deal of information about $R$, and there is a long tradition of using them when studying classes of rings. In the finitely generated case, by taking isomorphism classes, one can construct the commutative monoid $\V(R)$, where the operation is induced from the direct sum of modules. (Note that the Grothendieck group of $\V(R)$ is $\K_0(R)$.) If, instead, we focus on the class of countably generated projective modules, then we can consider the monoid $\V^*(R)$. In this case, one potentially gains access to significantly more information since, by the celebrated theorem of Kaplansky, any projective module is isomorphic to a direct sum of countably generated projective modules. Note also that, in this more general setting, the monoid $\V^*(R)$ contains an infinite element,  simply because $R^{(\mathbb N)}\oplus P\cong R^{(\mathbb N)}$ for any countably generated projective module $P$. Therefore, its Grothendieck group vanishes. This highlights the need to operate at the monoid level, admittedly losing some structure.
The monoid $\V^*(R)$ has been under intense scrutiny in various works; see, for example, \cite{HerPri2010}, \cite{HerPri2014}, \cite{HerPri2014JAlg}, \cite{HerPriWie2023}, and, for more general monoids, \cite{nazsme24monoid}. We remark that our arguments are almost exclusively concerned with order structure as described in the following rather than with the direct sum decompositions of the modules involved.

Both monoids $\V(R)$ and $\V^*(R)$ are usually ordered using the so-called algebraic preorder, given by complements.~Albeit this is natural -- the preorder on $\V(R)$ induces a natural preorder on $\K_0(R)$ - it might not be convenient as one can lose some of the fine relationships among projective modules. This is particularly relevant in the countably generated case and led the authors of \cite{AntAraBosPerVil25} to introduce a more general preorder and to call the resulting semigroup the Cuntz semigroup of the ring, denoted by $\CP(R)$ (or also by $\mathrm{S}(R)$, depending on the picture adopted). We briefly recall the preorder: for a unital ring $R$ and countably generated projective modules $P$, $Q$, we say that $P$ is Cuntz subequivalent to $Q$, and write $P\precsim Q$, if for any finitely generated submodule $X$ of $P$, one can factor the inclusion $\iota_X$ of $X$ in $P$ through convenient module morphisms $\phi\colon X\to Q$ and $\psi\colon Q\to P$. One then makes $\CP(R)$ a partially ordered semigroup by declaring countably generated projective modules $P$ and $Q$ to be Cuntz equivalent, in symbols $P\sim Q$, whenever $P\precsim Q$ and $Q\precsim P$; see also~\autoref{definicio cp(r) unital} below for more details.~Observe that if $P$ is finitely generated, then the relation $\precsim$ records the fact that $P$ itself is isomorphic to a direct summand of $Q$.~The definition given was inspired by the preorder defined on the isomorphism classes of countably generated Hilbert modules over a C*-algebra $A$, which turns out to agree with the Cuntz semigroup $\Cu(A)$ of $A$; see \cite{CowEllIva08CuInv}.~This is an object that plays a key role in the understanding of the internal structure of the algebra and, for the class of classifiable algebras, is functorially equivalent to the Elliott invariant; see \cite{AntDadPerSan14RecoverElliott} and \cite{AntPerRobThi22}.~Therefore, one expects this semigroup to encode a similar amount of information for well behaved classes of rings that help elucidate their internal structure. It was already shown that, in full generality, $\CP(R)$ captures the lattice of quasi-pure ideals of $R$, see \cite[Theorem 4.10]{AntAraBosPerVil26}.
In addition, the assignment $R\mapsto \CP(R)$ defines a continuous functor for the large class of the so-called left normal rings, see \cite[Theorem 8.2]{AntAraBosPerVil26}.

In this paper, we continue to pursue this program by analyzing the Cuntz semigroup of rings with stable rank one.~Briefly, a unital ring $R$ has stable rank one if, whenever $xa+b=1$ in $R$, there is $y\in R$ such that $a+yb$ is left invertible. It is well known that this is a symmetric condition, and it can be also defined for non-unital rings; see \autoref{def stable rank one} for more details.~Rings with stable rank one enjoy good permanence properties; namely, the condition is stable under matrix formation, corners, passage to inductive limits or surjective inverse limits, among others (a nice survey can be found in \cite{Vaserstein84StRangeCond}). This class of rings is pleasantly large and has been considered in many cases in the literature.~For example, it contains all semilocal rings \cite[Theorem 4.4]{facchini98}, all unit-regular rings \cite[Proposition 4.12]{goodearlvnrr}, group algebras over locally finite groups, strongly $\pi$-regular rings (\cite{Ara96Strongly}), Leavitt path algebras over acyclic graphs (see, e.g.~\cite[4.4]{AbrAraSil17Leavitt}), and many C*-algebras, such as AF-algebras, irrational rotation algebras, the Jiang-Su algebra, finite von Neumann algebras, and many others.~A prominent commutative example consists of the ring of continuous complex-valued functions over a compact Hausdorff space of dimension at most 1 \cite{vaserstein1971}. In fact, these are precisely the commutative unital C*-algebras that have stable rank one.

In the case of C*-algebras with stable rank one, it was proved in \cite{CowEllIva08CuInv} that the Cuntz subequivalence defined for countably generated Hilbert modules agrees with the relation of subisomorphism (i.e., being isomorphic to a submodule) and that Cuntz equivalence corresponds to isomorphism. The different nature of the Cuntz semigroups just mentioned (algebraic and analytic) was explored in \cite[Section 7]{AntAraBosPerVil25}. More concretely, it was shown that, for a C*-algebra $A$ and suitably interpreted, $\Cu(A)$ is a retract of $\CP(A)$. It remains an open problem to decide whether these semigroups agree in general. Taking into account these observations, it is pertinent to ask, for a ring with stable rank one, what the exact relationship is between the comparison relation $\precsim$ (respectively, $\sim$) and subisomorphism (respectively, isomorphism). This is the content of our first main result.

\begin{thmintro}[cf. \ref{main theorem}]
\label{thmintroA}
Let $R$ be an arbitrary ring with stable rank one, and let $P$, $Q$ be countably generated projective right $R$-modules. Then
\begin{enumerate}[label=\rm{(\roman*)}]
    \item $P\precsim Q$ if and only if $P$ is isomorphic to a pure submodule of $Q$.
    \item $P\sim Q$ if and only if $P\cong Q$.
\end{enumerate}
\end{thmintro}
Pureness will be discussed in more detail in \autoref{pgr:pure}. Suffice to say here that since the modules involved are projective, in this context it is equivalent to the following: $P$ is a pure submodule of $Q$ if, for any finitely generated submodule $X$ of $P$, there is $\sigma\colon Q\to P$ such that $\sigma_{|X}$ is the inclusion of $X$ in $P$.

Our result parallels very satisfactorily that of C*-algebras of stable rank one and applies to a much larger class of rings. It is also important to note that we do not have all the analytic machinery used in \cite{CowEllIva08CuInv} at our disposal. Therefore, the proof of the result necessarily breaks away from its C*-counterpart, although in spirit it still preserves a flavor of successive algebraic approximations. Our result immediately yields the following
\begin{corintro}
For any ring $R$ with stable rank one, the natural map $\V^*(R)\to\CP(R)$ is an isomorphism of commutative monoids. 
\end{corintro}
Since by \cite[Lemma~4.3, Theorem~4.10]{AntAraBosPerVil25}, the Cuntz semigroup $\CP(R)$ is closed under suprema of increasing sequences, the corollary above shows a way to interpret countable direct sums in $\V^*(R)$; see \cite{nazsme24monoid}. 

One of the classical results in the literature establishes that if $R$ is a unital ring with stable rank one, then finitely generated projective modules cancel from direct sums; see \cite[Theorem 2]{Evans1973}. A similar result for stable rank one C*-algebras was proved in \cite[Proposition 4.2]{RorWin10ZRevisited}, meaning that one can cancel finitely generated projective modules from inequalities in direct sums with general countably generated Hilbert modules. With a completely different approach, we show the following:

\begin{thmintro}[cf.~\ref{cancelacio fg}, \ref{cancelacio cp(r) fg}]
\label{thmintro B}
Let $R$ be a unital ring of stable rank one, let $P$, $Q$ be countably generated projective $R$-modules, and let $T$ be a finitely generated projective $R$-module. If $P\oplus T\precsim Q\oplus T$, then $P\precsim Q$. 
\end{thmintro}

We also explore more general forms of cancellation that do not need finite generation as discussed above. In the realm of C*-algebras, results of this type have been established in \cite[Theorem 4.3]{RorWin10ZRevisited}; see also \cite[Lemma 2.5]{AntPerRobThi22} for equivalent formulations in the context of abstract Cuntz semigroups. These results need the existence of certain complements, always present in C*-algebras (the so-called axiom of almost algebraic order), but this is not guaranteed. Thus, we need to develop an algebraic analog of a complement. This is done in \autoref{dfn:complement}. Loosely speaking, if $D\subseteq C\subseteq R^n$ are countably generated projective modules, then having a complement of $D\subseteq C$ ensures the existence of a countably generated projective module $E$ such that $D\oplus E\precsim R^n\precsim C\oplus E$. In general rings, complements need not exist, but we show that they do in some situations; see \autoref{prp existencia complement} and \autoref{cor:complement}. Using this, we can show the following
\begin{thmintro}[cf.~\ref{complement implica cancelacio}]
\label{thmintroC}
Let $R$ be a unital ring with stable rank one. 
Let $P,Q,C,D$ be countably generated projective modules such that $P\oplus C\precsim Q\oplus D$, with $D\subseteq C\subseteq R^n$ for some $n$ having a complement in $R^n$. Then $P\precsim Q$.
\end{thmintro}

The paper is organized as follows. In Section 2 we collect the notions and the necessary material needed for the rest of the results. In Section 3 we introduce the ideal $K(P)$ for any right module $P$ and show that $K(P)$ has stable rank one whenever $R$ also has stable rank one and $P$ is a countably generated projective module. In Section 4 we establish our first main result \autoref{main theorem}, which subsumes Theorem~\ref{thmintroA}. Section 5 is devoted to cancellation questions, and in particular we prove a somewhat more general version of Theorem~\ref{thmintro B}, as well as Theorem~\ref{thmintroC} along with some consequences.

We work as much as possible with arbitrary rings, that is, rings that do not necessarily have a unit, and pass to their Dorroh unitization when needed. For some of our results, and to keep technical issues at a minimum, we assume the existence of a unit.~All monoids in this paper are commutative, written additively, and we denote by $0$ their neutral element.

\section{Preliminaries}

In this section, we recall the definition of the main object in this paper. We also collect some notations used throughout the paper.

\begin{pgr}[The semigroup $\CP(R)$ for a unital ring $R$]\label{definicio cp(r) unital} Let $R$ be a unital ring and denote by $\mathcal{CP}(R)$ the class of all countably generated projective right $R$-modules. Given $P, Q \in \mathcal{CP}(R)$, we say that $P$ is \emph{Cuntz subequivalent to} $Q$, and write $P \precsim Q$ if, and only if, for every finitely generated submodule $X$ of $P$, there exists a factorization of the inclusion of $X$ in $P$ by $Q$, that is, there are module homomorphisms $\phi\colon X \rightarrow Q$ and $\psi \colon Q \rightarrow P$ such that $\psi \circ \phi = \mathrm{id}_X$,
\[
\begin{tikzcd}
X \arrow[r,"\phi"] \arrow[swap, rr, bend right, "\mathrm{id}_X"] & Q \arrow[r, "\psi"] & P 
\end{tikzcd}.
\]
Note that we are making the standard abuse of notation of identifying the identity map $\mathrm{id}_X$ with the natural inclusion $\iota \colon X \rightarrow P$.

We say that $P$ and $Q$ are \emph{Cuntz equivalent}, in symbols $P\sim Q$, provided $P \precsim Q$ and $Q \precsim P$. Define the partially ordered set $\mathrm{CP}(R)$ to be $\mathrm{CP}(R) \coloneqq \mathcal{CP}(R)/{\sim}$. Given an element $P \in \mathcal{CP}(R)$ denote its equivalence class by $[P]$. For modules $P, Q \in \mathcal{CP}(R)$ define $[P] + [Q] = [P \oplus Q]$. It was proved in \cite[Lemma 4.5]{AntAraBosPerVil25} that $\mathrm{CP}(R)$, equipped with the sum defined above and the order induced by $\precsim$, is a positively ordered commutative monoid. We call $\mathrm{CP}(R)$ the \emph{Cuntz semigroup of the ring} $R$. 
\end{pgr}
\begin{pgr}[Auxiliary relations]
Following \cite[Definition I-1.11]{GieHof+03Domains} or \cite[2.1.1]{AntPerThi14arX:TensorProdCu}, an \emph{auxiliary relation} on a partially ordered set $(P, \leq)$ is a binary relation on $P$ formally stronger than $\leq$ (i.e., $x \prec y \implies x \leq y$ for all $x, y \in P$) such that, for any $x',x,y,y' \in P$ with $x' \leq x \prec y \leq y'$, one has $x' \prec y'$. If, further, $P$ is also a monoid, the auxiliary relation is termed \emph{additive} if $0 \prec x$ for any $x \in P$ and if, whenever $x',x,y,y' \in P$ satisfy $x \prec y$ and $x' \prec y'$, we have $x + x' \prec y + y'$.

One of the most common examples of an auxiliary relation consists of the way-below relation. Assume that $P$ is closed under suprema of increasing sequences. If $x,y\in P$, we write $x\ll y$ if, whenever $(y_n)$ is an increasing sequence in $P$ such that $y\leq \sup y_n$, then there is $k$ for which $x\leq y_k$.  

We now define another relation $\prec$ on $\mathcal{CP}(R)$, stronger than $\precsim$, as follows: given $P, Q \in \mathcal{CP}(R)$, we write $P \prec Q$ if and only if there exists a finitely generated submodule $Y$ of $Q$ such that for every finitely generated submodule $X$ of $P$, there are module homomorphisms $\phi\colon X \rightarrow Q$ and $\psi \colon Q \rightarrow P$ such that $\psi \circ \phi = \mathrm{id}_X$, and moreover $\phi(X) \subseteq Y$. When we wish to emphasize the choice of the submodule $Y$, we write $P \prec_Y Q$ and say that $Y$ \emph{witnesses the relation} $P \prec Q$. By a slight abuse of notation, we write $\prec$ also for the relation on $\mathrm{CP}(R)$ induced by the relation $\prec$ on $\mathcal{CP}(R)$. It is immediate to verify that the relation $\prec$ on $\mathrm{CP}(R)$ is an additive auxiliary relation on $\mathrm{CP}(R)$ with the usual order.   

There is an alternative construction of the Cuntz semigroup, denoted by $\mathrm{S}(R)$, given in terms of certain equivalence classes of sequences of matrices over $R$; see \cite[4.1]{AntAraBosPerVil25}. Using \cite[Theorem 4.10]{AntAraBosPerVil25}, it can be shown that the relation $\prec$ on $\CP(R)$ agrees with the relation also denoted $\prec$ on $\mathrm{S}(R)$ given in \cite[Remark 6.4]{AntAraBosPerVil25}. It was shown in \cite[Lemma 4.3]{AntAraBosPerVil25} that $\CP(R)$ contains the supremum of any increasing sequence.~The nature of the relation $\prec$ on $\CP(R)$ makes it reasonable to relate it to the way-below relation. We know that $\prec$ is stronger than $\ll$, and they agree for stably left normal rings; see \cite[Theorem 7.8]{AntAraBosPerVil26} and the comments in \autoref{pgr:normal}.
\end{pgr}

\begin{pgr}[The semigroup $\mathrm{CP}(R)$ for a non-unital ring $R$] \label{def cp(r) no unital} Let $R$ be an arbitrary ring, denote by $R^{+}=\mathbb{Z}\oplus R$ the Dorroh extension of $R$, and view $R$ as a two-sided ideal of $R^+$. Denote by $\mathcal{CP}(R)$ the class of all countably generated unital projective right $R^{+}$-modules $P$ such that $P=PR$. Observe that this notation is consistent with the unital case. 
Denote by $\mathrm{FCM}(R)$ the ring of finite-column matrices over $R$. Given a module $P$ in $\mathcal{CP}(R)$, there exists an idempotent $e \in \mathrm{FCM}(R^+)$ such that $e((R^+)^{(\mathbb N)}) \cong P$.
Since $P = PR$, we necessarily have $e \in \mathrm{FCM}(R)$, and hence $e((R^+)^{(\mathbb N)}) = e(R^{(\mathbb N)})$. Conversely, every idempotent $e \in \mathrm{FCM}(R)$ gives rise to a countably generated projective $R^+$-module $P = e((R^+)^{(\mathbb N)}) = e(R^{(\mathbb N)})$, which satisfies $P = PR$. 


Observe that $\mathcal{CP}(R)$ is a subclass of $\mathcal{CP}(R^{+})$ closed under direct summands and countable direct sums. We equip $\mathcal{CP}(R)$ with the relation $\precsim$ inherited from $\mathcal{CP}(R^{+})$, and we write $\mathrm{CP}(R)$ for the monoid of equivalence classes of objects in $\mathcal{CP}(R)$ with respect to the relation $\precsim$. The inclusion $\mathcal{CP}(R) \hookrightarrow \mathcal{CP}(R^+)$ then induces an order-embedding $\mathrm{CP}(R) \hookrightarrow \mathrm{CP}(R^+)$, and $\mathrm{CP}(R)$ is a positively ordered monoid, called the \emph{Cuntz semigroup of the ring} $R$. Note that $\mathrm{CP}(R)$ agrees with the previously defined monoid whenever $R$ is a unital ring. 

We also define the relation $\prec$ on $\mathcal{CP}(R)$ by identifying it with the relation inherited from $\prec$ on $\mathcal{CP}(R^+)$, which induces an additive auxiliary relation on $\mathrm{CP}(R)$ with the usual order, as in the unital case. 
\end{pgr}

\section{The ideal $K(P)$}

In this section, we introduce $K(P)$ as an analog of the algebra of compact operators in a Hilbert space, but for a countably generated projective module $P$. We identify $K(P)$ in several fundamental cases and prove that, whenever the base ring has stable rank one, the ring $K(P)$ also has stable rank one. In analogy to \cite{CowEllIva08CuInv}, this result will be an important ingredient in the proof of the main theorem in the next section. 

\begin{pgr}[The group $K(P, Q)$] \label{def K(P)}
Let $R$ be a unital ring. Given a right $R$-module $P$, let $P^*=\mathrm{Hom}_R(P,R)$ denote its dual left $R$-module, and write $\langle \cdot, \cdot \rangle \colon P^* \times P \rightarrow R$ for the canonical evaluation pairing, given by $\langle \sigma, p \rangle = \sigma(p)$. Let $Q$ be another right $R$-module. For $q \in Q$ and $\sigma \in P^*$, we define the homomorphism $\theta_{q,\sigma} \colon P \rightarrow Q$ by $$\theta_{q,\sigma}(p) = q\langle\sigma,p\rangle.$$ 

We define $K(P,Q)$ to be the image of the group homomorphism
\begin{align*}
Q \otimes_R P^* & \longrightarrow \mathrm{Hom}_R(P, Q) \\
q \otimes \sigma & \longmapsto \theta_{q, \sigma}.
\end{align*}
Equivalently, $K(P,Q)$ is the subgroup of $\mathrm{Hom}_R(P, Q)$ generated by $\{ \theta_{q,\sigma} \colon \sigma \in P^*, q \in Q \}$, that is, 
\begin{equation*}
K(P, Q) = \left\{ \sum_{i = 1}^n \theta_{q_i, \sigma_i} \colon n \geq 1, q_i \in Q, \sigma_i \in P^* \right\}. 
\end{equation*}
Observe that the homomorphisms $\theta \in K(P, Q)$ are precisely the homomorphisms in $\mathrm{Hom}_R(P, Q)$ that factorize through $R^n$ for some $n \in \mathbb{N}$. That is, for $\theta \in \mathrm{Hom}_R(P, Q)$, we have that $\theta \in K(P, Q)$ if and only if there are module homomorphisms $\tau \colon R^n \rightarrow Q$ and $\rho \colon P \rightarrow R^n$ such that $\theta = \tau \circ \rho$.  

The following properties are easy to check: for $\theta_{q, \sigma} \in K(P,Q)$ and $\theta_{s, \tau} \in K(Q,S)$,
\begin{enumerate}[label=(\roman*)]
\item $\theta_{s, \tau}\circ\theta_{q,\sigma} = \theta_{s\tau(q), \sigma} \in K(P, S)$,
\item $f \circ \theta_{q,\sigma} = \theta_{f(q),\sigma} \in K(P, N)$ for $f \in \mathrm{Hom}_R(Q, N)$,
\item $\theta_{q,\sigma} \circ g = \theta_{q,\sigma\circ g} \in K(M, Q)$ for $g \in \mathrm{Hom}_R(M, P)$. 
\end{enumerate}

We define $K(P) \coloneqq K(P,P)$, a two-sided ideal of $\mathrm{End}_R(P)$ because of the above properties. 

Although $K(P)$ is defined for arbitrary modules, we shall work with countably generated projective modules.
\end{pgr}

Denote by $M_n(R)$ the ring of $n \times n$ matrices over $R$ and by $\mathrm{RFM}(R)$ the subring of matrices in $\mathrm{FCM}(R)$ that have a finite number of rows different from zero. 

Parts (i) and (ii) of the next proposition follow immediately from the fact that $K(P)$ consists of the endomorphisms of $P$ that factorize through $R^n$ for some $n$. Nevertheless, we give a direct argument. 

\begin{prp} \label{k(r^n)}
Let $R$ be a unital ring. Then we have:
\begin{enumerate}[label=\rm{(\roman*)}]
\item $K(R) \cong R$. 
\item $K(R^n) \cong M_n(R)$.
\item $K(R^{(\mathbb{N})}) \cong \mathrm{RFM}(R)$.
\end{enumerate}
\end{prp}

\begin{proof}
(i): The map sending $\sum_{i = 1}^n \theta_{p_i, \sigma_i}$ to $\sum_{i=1}^n p_i\sigma_i(1)$ provides the desired isomorphism. 

(ii): Here the map is given by the rule 
\begin{equation*}
\theta_{(p_1, \dots, p_n), (\sigma_1, \dots, \sigma_n)} \mapsto \begin{pmatrix}
    p_1\sigma_1(1) & \cdots & p_1\sigma_n(1) \\
    \vdots & \ddots & \vdots \\
    p_n\sigma_1(1) & \cdots & p_n\sigma_n(1) \\
\end{pmatrix}.
\end{equation*}
Proving that it is an isomorphism is routine. For example, to see surjectivity, observe that if 
\begin{equation*}
A = \begin{pmatrix}
    a_{11} & \cdots & a_{1n} \\
    \vdots & \ddots & \vdots \\
    a_{n1} & \cdots & a_{nn} \\
\end{pmatrix} \in M_n(R)
\end{equation*}
and $\theta_{ij} = \theta_{(0, \dots, 0, 1, 0, \dots, 0), (0, \dots, 0, \sigma_{ij}, 0, \dots, 0)}$, where $1$ is in the $i$-th position and $\sigma_{ij}$ is in the $j$-th position and is the map given by $\sigma_{ij}(1) = a_{ij}$, then  
\begin{equation*}
\sum_{1 \leq i, j \leq n} \theta_{ij} \mapsto A.
\end{equation*}
(iii): This is similar to (ii). Let $p = (p_i)_{i \in \mathbb{N}} \in R^{(\mathbb{N})}$ and $\sigma = (\sigma_i)_{i \in \mathbb{N}} \in (R^{(\mathbb{N})})^*$. Consider the map given by the rule 
\begin{equation*}
\theta_{p, \sigma} \mapsto \begin{pmatrix}
    p_1\sigma_1(1) & \cdots & p_1\sigma_i(1) & \cdots \\
    \vdots & \ddots & \vdots & \\
    p_i\sigma_1(1) & \cdots & p_i\sigma_i(1) & \cdots \\
    \vdots & & \vdots & \ddots \\
\end{pmatrix}.
\end{equation*}
Since $p_i = 0$ for all $i > N$ for some $N \in \mathbb{N}$, the matrix above is in $\mathrm{RFM}(R)$. A finite sum of matrices with finitely many non-zero rows again has only finitely many non-zero rows. Thus, every element of $K(R^{(\mathbb{N})})$ is mapped to a matrix in $\mathrm{RFM}(R)$. Seeing that the map is an isomorphism is again routine. Let us only check surjectivity. Let 
\begin{equation*}
A = \begin{pmatrix}
    a_{11} & \cdots & a_{1i} & \cdots \\
    \vdots & \ddots & \vdots &  \\
    a_{N1} & \cdots & a_{Ni} & \cdots \\
    0 & \cdots & 0 & \cdots \\
    \vdots &  & \vdots & \ddots \\
\end{pmatrix} \in \mathrm{RFM}(R).
\end{equation*}
Let $\theta_i \coloneqq \theta_{(0, \dots, 0, 1, 0, \dots, 0), (\sigma_{i1}, \sigma_{i2}, \dots, \sigma_{ij}, \dots)}$, where $1$ is in the $i$-th position and $\sigma_{ij}(1) = a_{ij}$ for all $1 \leq i \leq N$ and $j \geq 1$. Then 
\begin{equation*}
\sum_{1 \leq i \leq N} \theta_{i} \mapsto A.\qedhere
\end{equation*} 
\end{proof}

\begin{lma} \label{projectius i k(p)}
Let $R$ be a unital ring, $P$ be a projective right $R$-module and $Q$ any right $R$-module. Let $X$ be a finitely generated submodule of $P$ and $\phi \colon P \rightarrow Q$ be an $R$-module homomorphism. Then there exists $\theta \in K(P,Q)$ such that $\theta_{|X} = \phi_{|X}$. 
\end{lma}

\begin{proof}
Let $\{ (a_i, \sigma_i) \colon a_i \in P, \sigma_i \in P^*, i \in I\}$ be a dual basis for $P$, and let $\{x_1, \dots, x_n \}$ be a set of generators of $X$. For each $1\leq j\leq n$, there is a finite number of indices $i$ such that $\sigma_i(x_j)$ is non-zero. Let $J$ be the collection of all such indices. Then $J$ is a finite set and for all $x \in X$, $\sum_{i \in J} a_i\sigma_i(x) = x$. Then, consider $\theta \coloneqq \sum_{i \in J} \theta_{\phi(a_i), \sigma_i} \in K(P, Q)$, which satisfies that for all $x \in X$,
\begin{equation*}
\theta(x) = \sum_{i \in J} \theta_{\phi(a_i), \sigma_i}(x) = \sum_{i \in J} \phi(a_i)\sigma_i(x) = \phi (\sum_{i \in J} a_i\sigma_i(x)) = \phi(x).\qedhere
\end{equation*}
\end{proof}

\begin{rmk} \label{submoduls de projectius i k(p)}
In particular, if $P$ is a projective right $R$-module and $X$ is a finitely generated submodule of $P$, we can consider the identity map $\mathrm{id}\colon P \rightarrow P$ and by~ \autoref{projectius i k(p)}, there is some $\theta \in K(P)$ such that $\theta_{|X}$ is the inclusion morphism $\iota \colon X \rightarrow{} P$. 

\end{rmk}

\begin{prp}
Let $P$ be a projective right $R$-module over a unital ring $R$. Then, $P$ is countably generated if and only if there exists a sequence $(\theta_n) \subseteq K(P)$ such that $\theta_{n+1}\theta_n = \theta_n$ for every $n$, and such that for every $ \rho \in K(P)$ there exists $n$ satisfying $\theta_n\rho = \rho$. 
\end{prp}

\begin{proof}
Assume first that $P$ is countably generated, and let $\{x_1,x_2,\dots\}$ be a sequence of generators. By~\autoref{submoduls de projectius i k(p)}, for each $n$ there exists $\theta_n \in K(P)$ such that $\theta_n(x_i)=x_i$ for all $1 \le i \le n$. By possibly passing to a subsequence, we may assume that $\theta_{n+1}\theta_n=\theta_n$ for all $n$, as required. 

Conversely, suppose that there exists a sequence $(\theta_n) \subseteq K(P)$ satisfying the stated properties. Write each $\theta_n$ as a finite sum $\theta_n = \sum_i \theta_{p_{i,n},\sigma_{i,n}}$. Since $P$ is projective and the dual basis lemma holds for any projective module, we know that for every $x \in P$ there exists $\rho \in K(P)$ such that $\rho(x)=x$. By assumption, there exists $j$ such that $\theta_j \rho = \rho$, and hence
\[
x = \rho(x) = \theta_j \rho(x) \in \sum_i p_{i,j}R.
\]
Therefore, the set $\{p_{i,n}\}_{i,n}$ is a countable set of generators for $P$, and thus $P$ is countably generated.
\end{proof}

\begin{pgr}[Stable rank one for unital and non-unital rings] \label{def stable rank one}
Let $R$ be a unital ring. Recall that $R$ has \emph{stable rank one}, denoted $\mathrm{sr}(R)=1$, if for all $a,b \in R$ with $Ra + Rb = R$, there exists $y \in R$ such that $R(a + yb) = R$. Equivalently, if for all $a, b, x \in R$ with $xa + b = 1$, there exists $y \in R$ such that $a + yb$ is invertible in $R$. Note that this is the `left' version of the condition. There is a `right' version, and the two versions are in fact equivalent. 

An arbitrary ring $R$ is said to have stable rank one, again denoted $\mathrm{sr}(R) = 1$, if for any $a, b, x \in R^+$ with $a - 1 \in R$ and $xa + b = 1$, there exists $y \in R^+$ (equivalently, $y \in R$) such that $a + yb$ is invertible in $R^+$. This definition agrees with the previous one whenever $R$ is unital; see \cite[Theorem 3.4]{Vaserstein84StRangeCond}. 

Now suppose that $R$ is contained as a right ideal in a ring $R'$ with unit $1'$. Then $R$ has stable rank one if and only if, for any $a, b, x \in R'$ with $a - 1' \in R$ and $xa + b = 1'$, there exists $y \in R'$ (equivalently, $y \in R$) such that $a + yb$ is a unit in $R'$; see \cite[Lemma 3.5, Theorem 3.6]{Vaserstein84StRangeCond}.
\end{pgr}

The following proposition is immediate and well known, so we omit its proof.

\begin{prp} \label{rfm(r) two-sided fcm(r+)}
Let $R$ be an arbitrary ring. Then $\mathrm{RFM}(R)$ is a two-sided ideal of $\mathrm{FCM}(R^+)$.
\end{prp}

\begin{lma} \label{RFM(R) has stable rank one}
Let $R$ be an arbitrary ring with stable rank one. Then $\mathrm{RFM}(R)$ has stable rank one.     
\end{lma}

\begin{proof}
Since $\mathrm{RFM}(R)$ is a non-unital ring, to prove that $\mathrm{sr}(\mathrm{RFM}(R))=1$ one must use the notion of stable rank one for non-unital rings; see~\autoref{def stable rank one}. Due to~\autoref{rfm(r) two-sided fcm(r+)}, we know that $\mathrm{RFM}(R)$ is a two-sided ideal of the unital ring $\mathrm{FCM}(R^+)$. Thus, consider matrices $A, B, X \in \mathrm{FCM}(R^+)$ with $A - \Id \in \mathrm{RFM}(R)$ and $XA + B = \Id$, where $\Id$ is the identity matrix in $\mathrm{FCM}(R^+)$. We need to prove that there exists a matrix $Y \in \mathrm{FCM}(R^+)$ such that $A + YB$ is a unit in $\mathrm{FCM}(R^+)$.

Since $A - \Id \in \mathrm{RFM}(R)$, there is some $N \in \mathbb{N}$ such that all rows of the matrix $A - \Id$ from the $(N+1)$-th onward are zero. Write $A - \Id = (a_{i,j})_{i,j \in \mathbb{N}}$ and denote by $A_N$ its $N \times N$ submatrix in the upper left corner, that is, $A_N = (a_{i,j})_{1 \leq i,j \leq N}$. Put $A' = (a_{i,j})_{1 \leq i \leq N,\, j > N}$. We have
\begin{equation*}
    A = A - \Id + \Id = \begin{pmatrix}
        A_N & A' \\
        0 & 0 
    \end{pmatrix} + \begin{pmatrix}
        \Id_N & 0 \\
        0 & \Id_{\infty}
    \end{pmatrix},
\end{equation*}
where $\Id_N$ is the $N \times N$ identity matrix and $\Id_\infty$ is the infinite identity matrix. Similarly, put
\begin{equation*}
    X = \begin{pmatrix}
        X_N & X' \\
        X'' & X''' 
    \end{pmatrix}\quad \text{and} \quad B = \begin{pmatrix}
        B_N & B' \\
        B'' & B''' 
    \end{pmatrix},
\end{equation*}
where $X_N$ and $B_N$ are $N \times N$ matrices with coefficients in $R^+$. Then, equation $XA + B = \mathrm{Id}$ gives, in the upper-left block, the equation
$$X_N (A_N + \mathrm{Id}_N) + B_N = \Id_N.$$ 
Now, since $M_N(R)$ is a two-sided ideal of $M_N(R^+)$, $A_N + \Id_N - \Id_N = A_N \in M_N(R)$, and $M_N(R)$ has stable rank one (see \cite[Theorem 3.6]{Vaserstein84StRangeCond}), there exists some $Y_N \in M_N(R)$ such that $A_N + \Id_N + Y_NB_N$ is a unit in $M_N(R^+)$. If we define
\begin{equation*}
    Y = \begin{pmatrix}
        Y_N & 0 \\
        0 & 0 
    \end{pmatrix},
\end{equation*}
then 
\begin{equation*}
    A + YB = \begin{pmatrix}
        A_N + \Id_N + Y_NB_N & A' + Y_NB' \\
        0 & \Id_\infty 
    \end{pmatrix}
\end{equation*}
which is invertible in $\mathrm{FCM}(R^+)$ since each block on the diagonal is. Thus, $\mathrm{RFM}(R)$ has stable rank one.
\end{proof}

\begin{cor} \label{k(r^N) has stable rank one}
Let $R$ be a unital ring with stable rank one. Then $K(R^{(\mathbb{N})})$ has stable rank one. 
\end{cor}

\begin{proof}
This follows from~\autoref{k(r^n)} and~\autoref{RFM(R) has stable rank one}.
\end{proof}

The next lemma shows that the stable rank one property passes to corners of ideals, motivated by the classical result asserting that stable rank one passes to corners (see \cite[Theorem 2.8]{Vaserstein84StRangeCond}). This will be the key ingredient in proving that $K(P)$ has stable rank one whenever $R$ does and $P$ is a countably generated projective right $R$-module.

\begin{lma} \label{lema corner stable rank one}
Let $R$ be an arbitrary ring, $I$ a two-sided ideal of $R$ with $\mathrm{sr}(I) = 1$ and $e$ an idempotent in $R$. Then $\mathrm{sr}(eIe) = 1$. 
\end{lma}

\begin{proof}
It is easy to check that $I$ is a two-sided ideal of the unitization $R^+$ of $R$, so we can assume in the statement that $R$ is unital. 

Now, viewing $eIe$ as a two-sided ideal of the unital ring $eRe$ with identity $e$, we use the characterization discussed in~\autoref{def stable rank one} to prove that $\mathrm{sr}(eIe) = 1$. So let $a,b,x \in eRe$ such that $a - e \in eIe$ and $xa + b = e$. We must show that there exists $y \in eRe$ such that $a + yb$ is invertible in $eRe$. 

Define the following elements of $R$, $a' := a + (1-e), b' := b$ and $x' := x + (1-e).$ Since $x, a \in eRe$, 
\begin{equation*}
    x'a' + b' = (x + (1-e))(a + (1-e)) + b = xa + (1-e) + b = 1.
\end{equation*}
Observe that $a' - 1 = a + (1-e) - 1 = a - e \in eIe \subseteq I.$ Since $\mathrm{sr}(I)=1$, there exists $y' \in R$ such that
$u := a' + y'b' = a + (1-e) + y'b$
is a unit in $R$.

We now modify $y'$ to lie in $eRe$. Set $y := ey'e \in eRe$ and define $v := u(1 - (1-e)y'b)$. Since $((1-e)y'b)^2 = 0$, the element $1 - (1-e)y'b$ is a unit in $R$, hence $v$ is also a unit. Now, 
\begin{equation*}
v = (a + (1-e) + y'b)(1 - (1-e)y'b) = a + ey'b + (1 - e) = a + yb + (1-e).
\end{equation*}
Thus, $a + yb + (1-e)$ is invertible in $R$. Put $w \coloneqq (a + yb + (1-e))^{-1}$. 

Finally, let us show that $a + yb$ is a unit in $eRe$ whose inverse is $ewe$. Indeed, since $(a + yb + (1-e))w = 1$, right and left multiplication by $e$ yields $(a + yb)ewe = e$. Similarly, using $w(a + yb + (1-e)) = 1$ we obtain $ewe(a+yb) = e$. Hence, $\mathrm{sr}(eIe)=1.$
\end{proof}

We now combine the previous results to prove that the ring $K(P)$ inherits the stable rank one property from the underlying ring. The key observation is that $K(P)$ can be realized as a corner of $K(R^{\mathbb{(N)}})$.

\begin{prp} \label{k(p) te rang estable 1}
Let $R$ be an arbitrary ring with stable rank one and let $P \in \mathcal{CP}(R)$. Then $K(P)$ has stable rank one. 
\end{prp}

\begin{proof}
Since $P$ is a countably generated projective right $R^+$-module such that $PR = P$, there exists an idempotent $e \in \mathrm{FCM}(R)$ such that $P \cong e((R^{+})^{(\mathbb{N})}) = e(R^{(\mathbb{N})})$; see~\autoref{def cp(r) no unital}. Now, the map 
\begin{equation*}
    \mathrm{End}_R(e((R^+)^{(\mathbb{N})})) \rightarrow e\mathrm{End}_R((R^+)^{(\mathbb{N})})e 
\end{equation*}
given by $f \mapsto f \circ e$ defines a ring isomorphism and its restriction to $K(e((R^+)^{(\mathbb{N})}))$ takes values in $eK((R^+)^{(\mathbb{N})})e$. Indeed, 
let 
\begin{equation*}
    \theta = \sum_{i = 1}^n \theta_{e(x_i), \tau_i} \in K(e((R^+)^{(\mathbb{N})}))
\end{equation*}
with $x_i \in (R^+)^{(\mathbb{N})}$ and $\tau_i \in (e((R^+)^{(\mathbb{N})}))^*$. Then 
\begin{equation*}
    \theta \circ e = \sum_{i = 1}^n \theta_{e(x_i), \tau_i} \circ e = e \circ \sum_{i = 1}^n \theta_{x_i, \tau_i \circ e} \circ e,
\end{equation*}
and the latter belongs to $eK((R^+)^{(\mathbb{N})})e$, since $x_i \in (R^+)^{(\mathbb{N})}$ and $\tau_i\circ e \in ((R^+)^{(\mathbb{N})})^*$. Moreover, every element of $eK((R^+)^{(\mathbb{N})})e$ is the image of an element of 
$K(e((R^+)^{(\mathbb{N})}))$, since 
\begin{equation*}
    e \circ \sum_{i = 1}^n \theta_{x_i, \sigma_i} \circ e = \sum_{i = 1}^n \theta_{e(x_i), {\sigma_i}_{|e((R^+)^{(\mathbb{N})})}} \circ e,
\end{equation*}
where $x_i \in (R^+)^{(\mathbb{N})}$, $\sigma_i \in ((R^+)^{(\mathbb{N})})^*$, and ${\sigma_i}_{|e((R^+)^{(\mathbb{N})}}$ is the restriction of $\sigma_i$ to $e((R^+)^{(\mathbb{N})})$. Hence, $K(e((R^+)^{(\mathbb{N})}))$ and $eK((R^+)^{(\mathbb{N})})e$ are isomorphic as rings. 

Now, by~\autoref{rfm(r) two-sided fcm(r+)} we know that $\mathrm{RFM}(R)$ is a two-sided ideal of the unital ring $\mathrm{FCM}(R^+) \cong \mathrm{End}_{R^+}((R^+)^{(\mathbb{N})})$. Since $\mathrm{RFM}(R)$ has stable rank one (\autoref{RFM(R) has stable rank one}), and since $e$ is an idempotent in $\mathrm{FCM}(R^+)$, it follows from~\autoref{lema corner stable rank one} that $e\mathrm{RFM}(R)e$ also has stable rank one. 

Observe that $e\mathrm{RFM}(R)e = e\mathrm{RFM}(R^+)e$. Clearly $e\mathrm{RFM}(R)e \subseteq e\mathrm{RFM}(R^+)e$.~For the reverse inclusion, note that if $eae \in e\mathrm{RFM}(R^+)e$, with $a \in \mathrm{RFM}(R^+)$, then $ea \in \mathrm{RFM(R^+)}$ since $\mathrm{RFM(R^+)}$ is a two-sided ideal of $\mathrm{FCM}(R^+)$. But also, $ea$ has its coefficients in $R$, since $e \in \mathrm{FCM}(R)$ and $R$ is a two-sided ideal of $R^+$. Hence $ea \in \mathrm{RFM(R)}$. Therefore, $eae = e(ea)e \in e\mathrm{RFM}(R)e$. 

Putting everything together, we obtain
\begin{equation*}
    K(P) \cong K(e((R^+)^{(\mathbb{N})})) \cong eK((R^+)^{(\mathbb{N})})e \cong e\mathrm{RFM}(R^+)e = e\mathrm{RFM}(R)e,
\end{equation*}
where we have used~\autoref{k(r^n)} in the third step. Since the latter has stable rank one, we conclude that $K(P)$ has stable rank one. 
\end{proof}

\section{Cuntz equivalence in rings of stable rank one}

In this section, we prove the remarkable result that Cuntz equivalence of countably generated projective modules amounts to isomorphism whenever the ring has stable rank one. We begin by defining the notion of compact containment following the approach in \cite{CowEllIva08CuInv}. See also \cite[4.4]{AraPerTom11Cu} for a more detailed exposition.  

\begin{pgr}[Compact containment of modules]
Let $R$ be a unital ring, let $P$ be a right $R$-module, and let $X$ and $Y$ be submodules of $P$ with $X \subseteq Y$. We say that $X$ is \emph{compactly contained} in $Y$, denoted by $X \subset\!\subset Y$, if there exists $\theta \in K(P, Y)$ such that $\theta_{|X} = \mathrm{id}_{X}$, where $\mathrm{id}_X$ denotes the identity map of $X$. As mentioned before, note that we are making the standard abuse of notation of identifying the identity map $\mathrm{id}_X$ with the natural inclusion $\iota \colon X \rightarrow Y$. If the ambient module $P$ containing $X$ and $Y$ is not clear from the context, we shall refer to this situation by saying that $X \subset\!\subset Y$ in $P$.
\end{pgr}

\begin{rmk} \label{torre d'un projectiu en f.g. compactament continguts}
Let $P$ be a countably generated projective right module over a unital ring $R$. Then we can write $P$ as a countable union of finitely generated submodules, $$P = \bigcup_{i \geq 1} P_i$$ such that $P_i \subset\!\subset P_{i+1}$ for all $i \geq 1$. Indeed, let $\{ x_1, x_2, \dots\}$ be a countable generating set of $P$ and set $P_1 \coloneqq \langle x_1 \rangle$. Suppose that $P_n$ has been constructed. By~\autoref{submoduls de projectius i k(p)}, there exists $\theta \in K(P)$ such that $\theta_{|P_n} = \mathrm{id}_{P_n}$. Write $\theta = \sum_{i = 1}^m \theta_{y_i, \tau_i}$, where $y_i \in P$ and $\tau_i \in P^*$. Define $P_{n+1}$ to be the submodule generated by $x_1, \dots, x_{n+1}, y_1, \dots, y_m$. Then $P_{n+1}$ is finitely generated and $\theta \in K(P, P_{n+1})$. Hence, $P_n \cc P_{n+1}$.
\end{rmk}

The lemma below shows that, under the right assumptions, compact containment is preserved by module homomorphisms.

\begin{lma} \label{lema ucircphi CC ucircphi}
Let $R$ be a unital ring, $P$ and $Q$ in $\mathcal{CP}(R)$ and $A, B$ be submodules of $P$ such that $A \subset\!\subset B$. Suppose that there are module homomorphisms $\phi\colon B \rightarrow Q$ and $\psi\colon Q \rightarrow P$ such that $\psi \circ \phi_{|A} = \iota_A$, where $\iota_A$ is the inclusion map of $A$ in $P$.  Then $\phi(A) \subset\!\subset \phi(B)$. 
\end{lma}

\begin{proof}
Since $A \subset\!\subset B$, there exists $\theta \in K(P,B)$ such that $\theta_{|A} = \mathrm{id}_{A}$. Suppose 
\begin{equation*}
    \theta = \sum_{i=1}^n \theta_{b_i, \tau_i}
\end{equation*}
with $n \in \mathbb{N}$, $b_i \in B$ and $\tau_i \in P^*$. Set 
\begin{equation*}
    \theta' \coloneqq \sum_{i=1}^n \theta_{\phi (b_i), \tau_i \circ \psi}. 
\end{equation*}
Since for all $i \in \{1, \dots, n\}$, $\phi(b_i) \in \phi(B)$ and $\tau_i \circ \psi \in Q^*$, we see that $\theta' \in K(Q,\phi(B))$. Moreover, for all $a \in A$, we have 
\begin{equation*}
\theta'(\phi(a)) = \sum_{i=1}^n \phi(b_i)(\tau_i \circ \psi)(\phi(a)) = \sum_{i=1}^n \phi(b_i) \tau_i (a) = \phi(\theta(a)) = \phi(a), 
\end{equation*}
where we have used that $\psi \circ \phi_{|A} = \mathrm{\iota}_A$ in the second equality and that $\theta_{|A} = \mathrm{id}_{A}$ in the last equality. Hence, $\phi(A) \subset\!\subset \phi(B)$. 
\end{proof}

\begin{rmk} \label{auxiliar lemma cc}
Observe that, in the previous lemma, if we take $P = Q$, $\phi = u_{|B}$ and $\psi = u^{-1}$,  where $u \in \mathrm{Aut}_R(P)$, then $A \subset\!\subset B$ implies $u(A) \subset\!\subset u(B)$. 
\end{rmk}

The following lemma is the key to the proof of the main result in this section.

\begin{lma} \label{lema aprox}
Let $R$ be an arbitrary ring with stable rank one and let $P \in \mathcal{CP}(R)$. Suppose $A$, $B$, $A_1$, and $B_1$ are $R^+$-submodules of $P$ 
with $A_1 \subset\!\subset A$ and $B_1 \subset\!\subset B$. Let $Y$ be an $R^+$-submodule of $P$ containing $A$ and $B$ and let $\varphi \colon A \rightarrow B$ be an isomorphism such that $\varphi(A_1) = B_1$. Then, there exists an automorphism $u \in \mathrm{Aut}_{R^+}(P)$ such that $(u \circ \varphi)(s) = s$ for all $s \in A_1$, and such that $u(Y)\subseteq Y$, so that $(u \circ \varphi)(A) \subseteq Y$.  
\end{lma}

\begin{proof}
Let $\theta \in K(P, A)$ be such that $\theta_{|A_1} = \mathrm{id}_{A_1}$ and let $\theta' \in K(P, B)$ be such that $\theta'_{|B_1} = \mathrm{id}_{B_1}$. Put $\psi = \varphi^{-1}$, $\varphi' := \varphi \circ \theta$ and $\psi' := \psi \circ \theta'$. Since $A$ and $B$ are submodules of $Y$ we have that $\varphi' \in K(P, B) \subseteq K(P,Y)$ and $\psi' \in K(P, A) \subseteq K(P,Y)$. Put $x \coloneqq \psi' + 1 - \theta'$, $a \coloneqq \varphi' + 1 - \theta$ and $b \coloneqq 1 - (\psi' + 1 - \theta')(\varphi' + 1 - \theta)$, where $1$ is the identity map on $P$. Then   
\begin{equation*}
    xa + b = (\psi' + 1 - \theta')(\varphi' + 1 - \theta) + 1 - (\psi' + 1 - \theta')(\varphi' + 1 - \theta) = 1.
\end{equation*}
Now, observe that $K(P, Y)$ is a right ideal of $K(P)$, which has stable rank one because of~\autoref{k(p) te rang estable 1}. Hence, $K(P, Y)$ has stable rank one as well; see \cite[Theorem 3.6]{Vaserstein84StRangeCond}. 
Note also that $K(P,Y)$ is a right ideal of the unital ring $\mathrm{End}_{R^+}(P)$.
Since $x, a, b \in \mathrm{End}_{R^+}(P)$, $a - 1 = \varphi' - \theta \in K(P, Y)$ and $xa + b = 1$ by the above equation, the fact that $\mathrm{sr}(K(P, Y)) = 1$ implies that there exists $y \in K(P, Y)$ such that 
\begin{equation*}
    v \coloneqq (\varphi' + 1 - \theta) + y(1 - (\psi' + 1 - \theta')(\varphi' + 1 - \theta)) 
\end{equation*}
is an element of $\mathrm{Aut}_{R^+}(P)$. Observe that if $s \in A_1$, $\theta(s) = s$ and $\theta'(\varphi(s)) = \varphi(s)$, so we get, using that $\psi=\varphi^{-1}$, that
\begin{align*}
    (1 - (\psi' + 1 - \theta')(\varphi' + 1 - \theta))(s) & = s - (\psi' + 1 - \theta')(\varphi'(s) + s - \theta(s)) \\
    & = s - (\psi' + 1 - \theta')(\varphi(s)) \\
    & = s - (\psi \circ \varphi)(s) \\
    & = 0\; . 
\end{align*}
Hence, for all $s \in A_1$, $v(s) =\varphi'(s) + s - \theta(s) = \varphi(s)$. Put $u \coloneqq v^{-1}$. Then 
\begin{equation*}
    (u \circ \varphi)(s) = uv(s) = s,
\end{equation*}
as desired. Finally, to see that $(u \circ \varphi)(A) \subseteq Y$, observe that 
\begin{equation*}
    v = 1 + (\varphi' - \theta + y(1 - (\psi' + 1 - \theta')(\varphi' + 1 - \theta))),
\end{equation*}
so, since $\varphi', \theta, y \in K(P,Y)$, we may write $v = 1 + v'$, for some $v' \in K(P, Y)$. Then $v^{-1}(Y) \subseteq Y$. Otherwise, there is some $s \notin Y$ such that $v(s) \in Y$. But then
\begin{equation*}
    s = v(s) - v'(s) \in Y, 
\end{equation*}
which is a contradiction. Hence, $u(Y) \subseteq Y$, so $(u \circ \varphi)(A) \subseteq u(Y) \subseteq Y$, as we wanted to prove. 
\end{proof}

\begin{pgr}[Pure submodules]
\label{pgr:pure}
Let $R$ be a unital ring and let $\varphi \colon N \rightarrow M$ be a monomorphism of right $R$-modules. We say that $\varphi$ is a \emph{pure embedding} or a \emph{pure monomorphism} if for every left $R$-module $X$, the natural induced map
\begin{equation*}
    \varphi \otimes_R \mathrm{id}_X  \colon N \otimes_R X \longrightarrow M \otimes_R X
\end{equation*}
is injective. If $N$ is a submodule of $M$, we say that $N$ is a \emph{pure submodule} of $M$ if the natural inclusion map $\iota \colon N \hookrightarrow M$ is a pure embedding, and we write $N \leq_{\mathrm{pure}} M$. 

A monomorphism $\varphi \colon N \rightarrow M$ of right $R$-modules is said to be \emph{locally split} if, for any finitely generated submodule $F \subseteq N$, there exists $\sigma \in \mathrm{Hom}_R(M, N)$ such that $\sigma\circ\varphi_{|F} = \mathrm{id}_{F}$. If $\varphi \colon N \rightarrow M$ is a locally split monomorphism, then $\varphi(N)$ is a pure submodule of $M$ (see \cite[Exercise 4.38, p. 163]{lam1999lectures}). Moreover, if $M$ is projective and $\varphi \colon N \rightarrow M$ is a monomorphism of right $R$-modules, $\varphi$ is a pure embedding if and only if it is locally split; see \cite[Theorem 8]{Field72RegRings} and \cite[Proposition 2.2]{Chase1960}. In particular, $N$ is pure in $M$ if and only if the inclusion map $\iota \colon N \hookrightarrow M$ is locally split. 
\end{pgr}

\begin{prp} \label{impl right to left}
Let $R$ be a unital ring, and let $P$ and $Q$ be countably generated projective right $R$-modules such that $P$ is a pure submodule of $Q$. Then $P \precsim Q$. If, moreover, there is a finitely generated submodule $Y$ such that $P \subseteq Y \subseteq Q$, then $P \prec Q$. 
\end{prp} 

\begin{proof}
Let $X$ be a finitely generated submodule of $P$. Since $Q$ is projective, $P$ being a pure submodule of $Q$ is equivalent to the inclusion map $\iota\colon P \hookrightarrow Q$ being locally split. Hence, since $X$ is finitely generated, there exists $\sigma \in \mathrm{Hom}_R(Q, P)$ such that $\sigma_{|X} = \mathrm{id}_{X}$. Set $\phi \coloneqq \iota_{|X} \colon X \rightarrow Q$ and $\psi \coloneqq \sigma \colon Q \rightarrow P$. Then $\psi \circ \phi = \mathrm{id}_{X}$, so $P \precsim Q$ as desired. If, moreover, $Y$ is a finitely generated submodule such that $P \subseteq Y \subseteq Q$, then observe that $\phi(X) = X \subseteq P \subseteq Y$, so $P \prec Q$ as desired.  
\end{proof}

\begin{rmk}
\label{rmk:pures-are-projective} It is not hard to show, using \cite[Theorem 2.1]{Whitehead}, that a countably generated pure submodule of a projective module is itself a projective module. Therefore, 
in~\autoref{impl right to left} it is sufficient to assume that $P$ is countably generated and pure in $Q$ to conclude that $P$ is projective and $P\precsim Q$. 
Note also that this observation generalizes \cite[Corollary 2 of Theorem 10]{Field72RegRings} (see also \cite[Corollary 2.15]{goodearlvnrr}) to general rings. 
\end{rmk}

We are now in position to prove the main result of the paper. It provides a characterization of the order relation in $\mathrm{CP}(R)$ and shows that over rings of stable rank one, Cuntz equivalence of countably generated projective modules coincides with isomorphism. Parts of the proof follow the strategy of the analogous result of Coward, Elliott, and Ivanescu for Hilbert $C^*$-modules, see \cite[Theorem 3]{CowEllIva08CuInv}. For a more comprehensive treatment, see also \cite[Theorem 4.29]{AraPerTom11Cu}. 

Recall that if $f \colon A \rightarrow B$ is a map and $C$ is a subset of $B$ containing $f(A)$, the map $f^{|C} \colon A \rightarrow C$ given by $f^{|C}(a) = f(a)$ is called the corestriction of $f$ to $C$.

\begin{thm} \label{main theorem}
Let $R$ be an arbitrary ring with stable rank one and $P, Q \in \mathcal{CP}(R)$. Then 
\begin{enumerate}[label=\rm{(\roman*)}]
\item $P \prec Q$ if and only if $P \cong P' \leq_{\mathrm{pure}}Q$ and $P'$ is contained in a finitely generated submodule of $Q$. 
\item $[P] \leq [Q]$ in $\mathrm{CP}(R)$ if and only if $P$ is isomorphic to a pure submodule of $Q$.   
\item $[P] = [Q]$ in $\mathrm{CP}(R)$ if and only if $P \cong Q$. 
\end{enumerate}
\end{thm}

\begin{proof}
Write $P$ as a countable union of finitely generated submodules, $P = \bigcup_{i \geq 1} P_i$, such that $P_i \subset\!\subset P_{i+1}$ for all $i \geq 1$; see~\autoref{torre d'un projectiu en f.g. compactament continguts}. Similarly, write $Q = \bigcup_{i \geq 1} Q_i$ where each $Q_i$ is a finitely generated module and $Q_i \subset\!\subset Q_{i+1}$ for all $i \geq 1$. For notational convenience, we also assume that $P_0 = Q_0 = 0$. 

(i): The implication from right to left is a consequence of~\autoref{impl right to left}. For the converse implication, suppose that $P \prec_Y Q$, where $Y$ is a finitely generated submodule of $Q$. We will show that $P$ is isomorphic to a pure submodule $P'$ of $Q$ and that $P'$ is contained in $Y$. 

Since $P \prec_Y Q$ and, for each $i \geq 1$, $P_i$ is a finitely generated submodule of $P$, there exist module homomorphisms $\tilde{\phi}_i \colon P_i \rightarrow Q$, $\tilde{\psi}_i\colon Q \rightarrow P$ such that $\tilde{\psi}_i \circ \tilde{\phi}_i = \mathrm{id}_{P_i}$ and $\tilde{\phi}_i(P_i) \subseteq Y$. 



Now, we are going to construct inductively new morphisms $\phi_i \colon P_i \rightarrow Q$ for $i \geq 2$, modifying the old ones $\tilde{\phi}_i$ via $R^+$-automorphisms of $Q$. For each $i \geq 2$, $\phi_i$ will be of the form $u_i \circ \tilde{\phi}_i$ for some element $u_i \in \mathrm{Aut}_{R^+}(Q)$, with $u_i(Y)\subseteq Y$, and will satisfy that $\phi_{i}(P_i) \subseteq Y$
and that for all $x \in P_{i-2}$,
\begin{equation*} 
    \phi_{i}(x) = \phi_{i-1}(x).
\end{equation*}

Start by setting $\phi_1 = \tilde{\phi}_1$ and $\phi_2 = \tilde{\phi}_2$ (with $u_i=1$ for $i=1,2$), which satisfy the required conditions. Suppose that $\phi_1, \dots, \phi_n$ have been constructed for $n \geq 2$ and let us construct $\phi_{n+1}$. Since $\phi_n$ is of the form $u_n \circ \tilde{\phi}_n$ for some $u_n \in \mathrm{Aut}_{R^+}(Q)$ and $\tilde{\phi}_n$ is injective, we see that $\phi_n$ is also injective, so the map $$\varphi \coloneqq \tilde{\phi}_{n+1}\circ\phi_{n}^{-1} \colon \phi_n(P_n) \rightarrow \tilde{\phi}_{n+1}(P_n)$$ is an isomorphism. Set $A \coloneqq \phi_n(P_n)$ and $B \coloneqq \tilde{\phi}_{n+1}(P_n)$, which are submodules of $Q$, and let $A_1 \coloneqq \phi_n(P_{n-1}) \subseteq \phi_n(P_n) = A$. By~\autoref{lema ucircphi CC ucircphi} and since $P_{n-1} \subset\!\subset P_n$, we have $\tilde{\phi}_n(P_{n-1}) \subset\!\subset \tilde{\phi}_{n}(P_n)$, and by~\autoref{auxiliar lemma cc}, also $u_n \circ \tilde{\phi}_n(P_{n-1}) \subset\!\subset u_n \circ \tilde{\phi}_{n}(P_n)$, since $u_n$ is invertible. Hence, $A_1 \subset\!\subset A$. Now, define $B_1 \coloneqq \tilde{\phi}_{n+1}(P_{n-1})$. Observe that $\varphi(A_1) = \tilde{\phi}_{n+1}(\phi_{n}^{-1}(\phi_n(P_{n-1}))) = \tilde{\phi}_{n+1}(P_{n-1}) = B_1$. Moreover, by~\autoref{lema ucircphi CC ucircphi} again, $\tilde{\phi}_{n+1}(P_{n-1}) \subset\!\subset \tilde{\phi}_{n+1}(P_{n})$, so $B_1 \subset\!\subset B$. Moreover, $\phi_n(P_n) \subseteq Y$ by induction hypothesis, and $\tilde{\phi}_{n+1}(P_n) \subseteq Y$, that is, $A$ and $B$ are contained in the submodule $Y$. Hence, we may apply~\autoref{lema aprox} to $\varphi \colon A \rightarrow B$, $A_1 \subset\!\subset A$, $B_1  \subset\!\subset B$, $A, B \subseteq Y \subseteq Q$, so there exists an automorphism $u_{n+1} \in \mathrm{Aut}_{R^+}(Q)$ such that $u_{n+1}(Y)\subseteq Y$ and
\begin{equation} \label{equacio 2 teorema}
u_{n+1} \circ \tilde{\phi}_{n+1}\circ{\phi_{n}^{-1}}_{|\phi_n(P_{n-1})} = \mathrm{id}_{\phi_n(P_{n-1})}.    
\end{equation}
If $x \in P_{n-1}$ and $y \coloneqq \phi_n(x)$, by equation (\ref{equacio 2 teorema}) we have
\begin{equation*}
u_{n+1} \circ \tilde{\phi}_{n+1}(x) = u_{n+1} \circ \tilde{\phi}_{n+1}(\phi_n^{-1}(y)) = y = \phi_n(x).
\end{equation*} 
Hence, if we define $\phi_{n+1} \coloneqq u_{n+1} \circ \tilde{\phi}_{n+1}$, then for all $x \in P_{n-1}$, $\phi_{n+1}(x) = \phi_n(x)$ and also 
$$\phi_{n+1}(P_{n+1}) =u_{n+1}(\tilde{\phi}_{n+1}(P_{n+1}))
\subseteq u_{n+1}(Y)\subseteq Y,$$ 
as we wanted to prove.

Once the new homomorphisms $\phi_i$ have been constructed for all $i \geq 2$, we define a global homomorphism $\phi \colon P \rightarrow Q$ as follows. If $x \in P_i$, set $\phi(x) \coloneqq \phi_{i+1}(x)$. Observe that, if $x \in P_i$, $\phi_{i+2}(x) = \phi_{i+1}(x)$, so for all $j \geq i$, $\phi_{j+1}(x) = \phi(x)$. Hence, $\phi$ is a well-defined homomorphism. Since $\phi_i$ is injective for all $i \geq 1$, $\phi$ is also injective, so $P$ is isomorphic to the submodule $\phi(P)$ of $Q$. Moreover, since for all $i \geq 2$, $\phi_i(P_i)  \subseteq Y$, $\phi(P)$ is also a submodule of $Y$. 

Finally, we need to show that $\phi(P)$ is a pure submodule of $Q$. Consider a finitely generated submodule $X$ of $\phi(P)$. Since
\begin{equation*}
    \phi(P) = \phi\left(\bigcup_{i\geq1} P_i\right) =  \bigcup_{i\geq 1} \phi_{i+1}(P_i),
\end{equation*}
and $\phi_{i+1}(P_i) \subseteq\phi_{i+2}(P_{i+1})$, there exists $n$ such that $X \subseteq \phi_{n+1}(P_n)$. Recall that $\phi_{n+1} = u_{n+1} \circ \tilde{\phi}_{n+1}$ for some $u_{n+1}$ in $\mathrm{Aut}_{R^+}(Q)$. Now, set $\sigma \coloneqq \phi \circ \tilde{\psi}_{n+1}\circ u_{n+1}^{-1} \colon Q \rightarrow \phi(P)$. Then if $x \in X$, there is some $y \in P_n$ such that $x = \phi_{n+1}(y)$, and
\begin{equation*}
    \sigma(x) = \sigma\circ\phi_{n+1}(y) = \phi \circ \tilde{\psi}_{n+1}\circ u_{n+1}^{-1} \circ u_{n+1} \circ \tilde{\phi}_{n+1} (y) = \phi(y) = \phi_{n+1}(y) = x,
\end{equation*}
so $\sigma_{|X} = \mathrm{id}_{X}$. Since $\iota \colon \phi(P) \hookrightarrow Q$ is locally split, we conclude that $\phi(P)$ is pure in $Q$.

(ii): The right-to-left implication follows from~\autoref{impl right to left}, and the left-to-right implication is a consequence of the proof of part (i). Indeed, we did not use in (i) that $Y$ is a finitely generated submodule of $Q$, and exactly the same proof holds for an arbitrary submodule $Y$ of $Q$ (in particular for $Y=Q$). 

(iii): The implication from right to left is clear. For the converse, assume $[P] = [Q]$. Similarly to (i), we will construct inductively injective homomorphisms 
\begin{equation*}
    \phi_n \colon P_{4n} \rightarrow Q_{4n + 1} \quad \text{ and } \quad \alpha_n \colon Q_{4n + 2} \rightarrow P_{4n + 3}
\end{equation*}
for each $n \geq 1$. Also, for each $n$, we will have auxiliary homomorphisms $\tilde{\phi}_n \colon P_{4n} \rightarrow Q$ and $\tilde{\psi}_n \colon Q \rightarrow P$ and an $R^+$-automorphism $u_n$ of $Q$ such that $\tilde{\psi}_n \circ \tilde{\phi}_n = \iota_{P_{4n}}$, the inclusion of $P_{4n}$ into $P$, and $\phi_n$ will be the corestriction of $u_n \circ \tilde{\phi_n}$ to $Q_{4n + 1}$. Similarly, for each $n$, we will have homomorphisms $\tilde{\alpha}_n \colon Q_{4n+2} \rightarrow P$ and $\tilde{\beta}_n \colon P \rightarrow Q$ and an $R^+$-automorphism $v_n$ of $P$ such that $\tilde{\beta}_n \circ \tilde{\alpha}_n = \iota_{Q_{4n +2}}$, the inclusion of $Q_{4n+2}$ in $Q$, and $\alpha_n$ will be the corestriction of $v_n \circ \tilde{\alpha_n}$ to $P_{4n + 3}$. 

Start with $n = 1$. Since $P_4$ is a finitely generated submodule of $P$ and $P \precsim Q$, there exist module homomorphisms $\tilde{\phi}_1 \colon P_4 \rightarrow Q$ and $\tilde{\psi}_1 \colon Q \rightarrow P$ such that $\tilde{\psi}_1 \circ \tilde{\phi}_1 = \iota_{P_4}$. Since $\tilde{\phi}_1(P_4)$ is finitely generated, we may assume $\tilde{\phi}_1(P_4) \subseteq Q_5$. Define $\phi_1 \coloneqq \tilde{\phi}_1^{|Q_5}$, which satisfies the desired conditions (with $u_1=1$). 

Suppose that $\phi_1, \alpha_1, \phi_2, \dots, \phi_{n-1}, \alpha_{n-1}$ and $\phi_{n}$ and the corresponding auxiliary maps have been constructed, and let us construct $\alpha_{n}$ and $\phi_{n+1}$. We start with $\alpha_n$. Since $Q_{4n + 2}$ is a finitely generated submodule of $Q$ and $Q \precsim P$, there exist module homomorphisms $\tilde{\alpha}_n \colon Q_{4n + 2} \rightarrow P$ and $\tilde{\beta}_n \colon P \rightarrow Q$ such that $\tilde{\beta}_n \circ \tilde{\alpha}_n = \mathrm{id}_{Q_{4n + 2}}$. So far, the already constructed homomorphisms $\phi_i$ and $\alpha_i$ use the submodules $P_i$ with $i \leq 4n$, so we can assume, omitting possibly some indices, that $\tilde{\alpha}_n(Q_{4n + 2}) \subseteq P_{4n + 2}$. Similar to (i), we consider the isomorphism 
\begin{equation*}
    \tilde{\alpha}_n\circ\phi_n \colon P_{4n} \rightarrow (\tilde{\alpha}_n\circ\phi_n)(P_{4n}).
\end{equation*}

By construction, there are homomorphisms $\tilde{\phi}_n \colon P_{4n} \rightarrow Q$ and $\tilde{\psi}_n \colon Q \rightarrow P$ and an automorphism $u_n \in \mathrm{Aut}_{R^+}(Q)$ such that $\tilde{\psi}_n \circ \tilde{\phi}_n = \iota_{P_{4n}}$ and $\phi_n = (u_n \circ \tilde{\phi_n})^{|Q_{4n + 1}}$. Using~\autoref{lema ucircphi CC ucircphi} and~\autoref{auxiliar lemma cc} and since $P_{4n - 1} \subset\!\subset P_{4n}$, we have that $$\tilde{\alpha}_n\circ u_n \circ \tilde{\phi_n}(P_{4n - 1}) \subset\!\subset \tilde{\alpha}_n\circ u_n \circ \tilde{\phi_n}(P_{4n}).$$ 
Set $A_1 \coloneqq P_{4n - 1}$, $A \coloneqq P_{4n}$, $B_1 \coloneqq \tilde{\alpha}_n\circ\phi_n(P_{4n - 1})$, $B \coloneqq \tilde{\alpha}_n\circ\phi_n(P_{4n})$ and $\varphi \coloneqq \tilde{\alpha}_n\circ\phi_n \colon A \rightarrow B$. Then $\varphi$ is an isomorphism, $A_1 \cc A$, $B_1 \cc B$ and $\varphi(A_1) = B_1$, so using~\autoref{lema aprox}, there exists $v_n \in \mathrm{Aut}_{R^+}(P)$ such that
\begin{equation} \label{eq1 main teo}
{v_n \circ \tilde{\alpha}_n \circ \phi_n}_{|P_{4n-1}} = \mathrm{id}_{P_{4n - 1}}. 
\end{equation}
Now, since $v_n \circ \tilde{\alpha}_n(Q_{4n+2}) \subseteq P$ is finitely generated, we may assume, omitting possibly some indices, that it is contained in $P_{4n + 3}$. Then, we define $\alpha_n \coloneqq (v_n \circ \tilde{\alpha}_n)^{|P_{4n + 3}}$. Hence, $\alpha_n \colon Q_{4n + 2} \to P_{4n + 3}$ is an injective homomorphism of the desired form. 

To construct $\phi_{n+1}$, we proceed similarly. Since $P_{4n + 4}$ is a finitely generated submodule of $P$ and $P \precsim Q$, there exist module homomorphisms $\tiphi_{n+1} \colon P_{4n + 4} \rightarrow Q$ and $\tipsi_{n+1} \colon Q \rightarrow P$ such that $\tipsi_{n+1} \circ \tiphi_{n+1} = \iota_{P_{4n+4}}$. Since $\tiphi_{n+1}(P_{4n+4})$ is finitely generated, we can assume that it is contained in $Q_{4n+4}$. Now, consider the isomorphism 
\begin{equation*}
    \tiphi_{n+1} \circ \alpha_n \colon Q_{4n+2} \rightarrow \tiphi_{n+1} \circ \alpha_n(Q_{4n + 2}).
\end{equation*}
Applying~\autoref{lema ucircphi CC ucircphi} and~\autoref{auxiliar lemma cc}, together with the fact that $Q_{4n + 1} \cc Q_{4n + 2}$, we obtain $\tiphi_{n+1}\circ\alpha_n(Q_{4n + 1}) \cc \tiphi_{n+1}\circ\alpha_n(Q_{4n + 2})$. Hence,~\autoref{lema aprox} provides the  existence of an element $u_{n+1} \in \mathrm{Aut}_{R^+}(Q)$ such that 
\begin{equation} \label{eq2 main teo}
    {u_{n+1}\circ \tiphi_{n+1} \circ \alpha_n}_{|Q_{4n+1}} = \mathrm{id}_{Q_{4n+1}}.
\end{equation}
Since $u_{n+1} \circ \tiphi_{n+1}(P_{4n+4}) \subseteq Q$ is finitely generated, we may assume, omitting possibly some indices, that it is contained in $Q_{4n + 5}$. We define $\phi_{n+1} \coloneqq (u_{n+1} \circ \tiphi_{n+1})^{|Q_{4n + 5}}$. Then, $\phi_{n+1} \colon P_{4n + 4} \to Q_{4n + 5}$ is an injective homomorphism as required. 

We summarize the situation in the following (not necessarily commutative) diagram:
\[
\begin{tikzcd}[column sep=3.5em, row sep=4em]
\cdots \arrow[r, ""] &  P_{4n+2} \arrow[r, "\cc"] & P_{4n+3} \arrow[r, "\cc"] & P_{4n + 4} \arrow[d, "\tiphi_{n+1}^{|Q_{4n + 4}}"] \arrow[dr, "\phi_{n+1}"] & & \\
\cdots \arrow[r, ""] & Q_{4n+2} \arrow[u, "\tialpha_{n}^{|P_{4n + 2}}"] \arrow[ur, "\alpha_{n}"] &  & Q_{4n + 4} \arrow[r, "\cc"] & Q_{4n + 5} \arrow[r, "\cc"] & Q_{4n + 6}
\end{tikzcd}
\] 

To complete the proof, we construct homomorphisms $\phi \colon P \to Q$ and $\alpha \colon Q \to P$ from the previously defined 
$\phi_i$'s and $\alpha_i$'s, and show that
\begin{equation*}
 \phi \circ \alpha = \mathrm{id}_Q \quad \text{and} \quad
\alpha \circ \phi = \mathrm{id}_P.
\end{equation*}

From equations (\ref{eq1 main teo}) and (\ref{eq2 main teo}) we have that 
\begin{equation*}
{\alpha_n \circ \phi_n}_{|P_{4n - 1}} = \mathrm{id}_{P_{4n - 1}} \quad \text{and} \quad {\phi_{n+1} \circ \alpha_n}_{|Q_{4n + 1}} = \mathrm{id}_{Q_{4n + 1}}, 
\end{equation*}
and then, 
\begin{equation*}
    {\phi_{n+1}}_{|P_{4n - 1}} = {\phi_{n+1}\circ\alpha_n\circ\phi_n}_{|P_{4n - 1}} = {\phi_n}_{|P_{4n - 1}}.
\end{equation*}
Hence, similar to (i), the map $\phi \colon P \rightarrow Q$ given by $\phi(x) = \phi_{n}(x)$ if $x \in P_{4n - 1}$ is a well-defined homomorphism. Analogously, we have that
\begin{equation*}
    {\alpha_{n+1}}_{|Q_{4n + 1}} = {\alpha_{n+1}\phi_{n+1}\alpha_n}_{|Q_{4n + 1}} = {\alpha_{n}}_{|Q_{4n + 1}},
\end{equation*}
so the map $\alpha \colon Q \rightarrow P$ given by $\alpha(y) = \alpha_{n}(y)$ for $y \in Q_{4n+1}$ is a well-defined homomorphism. 

Now, if $x \in P_{4n -1}$, $\phi(x) = \phi_{n}(x) \in Q_{4n + 1}$, so $\alpha\circ\phi(x) = \alpha_n(\phi_n(x)) = x$. Hence, $\alpha \circ \phi = \mathrm{id}_{P}$. Similarly, one obtains that  $\phi \circ \alpha = \mathrm{id}_Q$. This shows $P \cong Q$ and concludes the proof. 
\end{proof}

\section{Cancellation properties}

Stable rank one is closely related to cancellation properties of projective modules. In this section, we prove cancellation results, some of which involve applications of \autoref{main theorem}.

Although our primary interest is the relation $\precsim$ on countably generated projective modules, note that Cuntz comparison, as we have defined, extends naturally to arbitrary modules. More precisely, given right $R$-modules $P$ and $Q$, we write $P \precsim Q$ if, for every finitely generated submodule $X$ of $P$, there exist module homomorphisms $\phi\colon X \rightarrow Q$ and $\psi\colon Q \rightarrow P$ such that $\psi\circ\phi = \mathrm{id}_X$. When restricted to $\mathcal{CP}(R)$, this obviously agrees with the original definition. Similarly, we may extend the definition of the relation $\prec$ to arbitrary right $R$-modules. The following theorem is in the spirit of Evans' classical cancellation theorem, see \cite[Theorem 2]{Evans1973}.

\begin{thm} \label{cancelacio fg}
Let $R$ be an arbitrary ring, $P$ and $Q$ be right $R$-modules and $T$ be a finitely generated right $R$-module with $\mathrm{sr}(\mathrm{End}_R(T)) = 1$. If 
\begin{equation*}
    P \oplus T \precsim Q \oplus T
\end{equation*}
then $P \precsim Q$. 
\end{thm}

\begin{proof}
Let $X$ be a finitely generated submodule of $P$. Then $X \oplus T$ is a finitely generated submodule of $P \oplus T$, so there exist module homomorphisms 
\begin{equation*}
    \phi \colon X \oplus T \rightarrow Q \oplus T, \quad \psi \colon Q \oplus T \rightarrow P \oplus T
\end{equation*}
with $\psi \circ \phi = \iota_{X} \oplus \mathrm{id}_{T}$, where $\iota_X \colon X \rightarrow P$ is the natural inclusion and $\mathrm{id}_T$ the identity map of $T$. Write $\phi$ and $\psi$ as $2 \times 2$ matrices 
\begin{equation*}
    \phi = \begin{pmatrix}
        \phi_1 & \phi_2 \\
        \phi_3 & \phi_4 
    \end{pmatrix}, \quad \psi = \begin{pmatrix}
        \psi_1 & \psi_2 \\
        \psi_3 & \psi_4 
    \end{pmatrix},
\end{equation*}
where $\phi_1\colon X \rightarrow Q$, $\phi_2\colon T \rightarrow Q$, $\phi_3\colon X \rightarrow T$, $\phi_4\colon T \rightarrow T$, $\psi_1\colon Q \rightarrow P$, $\psi_2\colon T \rightarrow P$, $\psi_3\colon Q \rightarrow T$, and $\psi_4\colon T \rightarrow T$. The condition $\psi \circ \phi = \iota_{X} \oplus \mathrm{id}_{T}$ gives
the following four identities
\begin{align} \label{eq1}
    \psi_1\phi_1 + \psi_2\phi_3 & = \iota_{X} \\
    \psi_3\phi_1 + \psi_4\phi_3 & = 0 \\ \label{eq3}
    \psi_1\phi_2 + \psi_2\phi_4 & = 0 \\ \label{eq4}
    \psi_3\phi_2 + \psi_4\phi_4 & = \mathrm{id}_{T} 
\end{align}
We now show that we can modify $\phi$ and $\psi$ so that $\phi_4$ is invertible. Set $x \coloneqq \psi_4, a \coloneqq \phi_4$ and $b \coloneqq \psi_3\phi_2$. Then $a, b, x \in \mathrm{End}_R(T)$ and from equation (\ref{eq4}) we have $xa + b = 1$, where 1 is the identity map on $T$. Since $\mathrm{sr}(\mathrm{End}_R(T)) = 1$, there exists $y \in \mathrm{End}_R(T)$ such that $a + yb$ is invertible in $\mathrm{End}_R(T)$. Define an automorphism $\tau \in \mathrm{Aut}_R(Q \oplus T)$ given by 
\begin{equation*}
    \tau = \begin{pmatrix}
        \mathrm{id}_Q & 0 \\
        y\psi_3 & \mathrm{id}_T 
    \end{pmatrix}, \quad \tau^{-1} =  \begin{pmatrix}
        \mathrm{id}_Q & 0 \\
        -y\psi_3 & \mathrm{id}_T
    \end{pmatrix}. 
\end{equation*}
Now replace $\phi$ by $\phi' \coloneqq \tau \circ \phi$ and $\psi$ by $\psi' \coloneqq \psi \circ \tau^{-1}$. Then $\psi' \circ \phi' = \iota_{X} \oplus \mathrm{id}_T$ still holds and 
\begin{equation*}
    \phi' = \begin{pmatrix}
        \phi_1 & \phi_2 \\
        y\psi_3\phi_1 + \phi_3 & y\psi_3\phi_2 + \phi_4 
    \end{pmatrix}.
\end{equation*}
The lower right corner of $\phi'$ is $\phi_4 + y\psi_3\phi_2 = a + yb$ which is invertible, as we wanted to prove. 

Hence, we may assume that $\phi_4$ is invertible in the original equations (\ref{eq1})-(\ref{eq4}). Now, from equation (\ref{eq3}) we have $\psi_2 = - \psi_1\phi_2\phi_4^{-1}$. Substituting this into equation (\ref{eq1}) yields
\begin{equation*}
    \iota_X = \psi_1\phi_1 - \psi_1\phi_2\phi_4^{-1}\phi_3 = \psi_1(\phi_1 - \phi_2\phi_4^{-1}\phi_3).
\end{equation*}
Thus, the diagram
\begin{equation*}
\begin{tikzcd}[column sep=huge]
X \arrow[r, "\phi_1 - \phi_2\phi_4^{-1}\phi_3"] \arrow[rr, bend right=30, "\iota_X"'] & Q \arrow[r, "\psi_1"] & P
\end{tikzcd}
\end{equation*}
commutes. Since $X$ was an arbitrary finitely generated submodule of $P$, this proves that $P \precsim Q$.
\end{proof}

As an immediate consequence of the last proposition, we obtain the following cancellation result for the Cuntz semigroup of a ring with stable rank one. There is an analogous result for the Cuntz semigroup of a $C^*$-algebra of stable rank one that was established, with different techniques, by R\o rdam and Winter in \cite[Proposition 4.1]{RorWin10ZRevisited}. 

\begin{cor} \label{cancelacio cp(r) fg}
Let $R$ be a unital ring with stable rank one, let $P$ and $Q$ be right $R$-modules and let $T$ be a finitely generated projective right $R$-module. Then $P \oplus T \precsim Q \oplus T$ implies $P \precsim Q$.  
\end{cor}

\begin{proof}
This follows immediately from the fact that $T$ is finitely generated and projective, which implies $\mathrm{sr}(\mathrm{End}_R(T)) = 1$, together with~\autoref{cancelacio fg}.
\end{proof}

We note that our result also give an alternative proof of \cite[Proposition 4.2]{RorWin10ZRevisited}. We briefly recall the construction of the Cuntz semigroup of a C*-algebra $A$. Given positive elements $a,b\in A$, one says that $a$ is \emph{Cuntz subequivalent} to $b$, in symbols $a\precsim b$ if, for any $\varepsilon>0$ there is $x\in A$ such that $\Vert a-xbx^*\Vert<\varepsilon$. One also says that $a$ and $b$ are \emph{Cuntz equivalent}, in symbols $a\sim b$, provided $a\precsim b$ and $b\precsim a$. The Cuntz semigroup of $A$ is defined as $\Cu(A)=(A\otimes\mathcal{K})_+/{\sim}$, where $A\otimes \mathcal{K}$ is the stabilization of $A$. The classes of elements in $\Cu(A)$ are denoted by $[a]$, where $a\in (A\otimes\mathcal{K})_+$.
\begin{cor}
Let $A$ be a C*-algebra of stable rank one. Let $x,y\in \Cu(A)$ and $[p]\in \Cu(A)$ with $p$ a projection in $A\otimes \mathcal{K}$. If $x+[p]\leq y+[p]$, then $x\leq y$.
\end{cor}
\begin{proof}
By (the proof of) \cite[Theorem 7.6]{AntAraBosPerVil25}, there are order-preserving semigroup morphisms $\varphi\colon \Cu(A)\to \CP(A)$ and $\psi\colon \CP(A)\to \Cu(A)$ such that $\psi\circ\varphi=\mathrm{id}_{\Cu(A)}$. Further, $\varphi([p])$ is the class of a finitely generated projective $A$-module. Therefore, if $x+[p]\leq y+[p]$, after applying $\varphi$ and invoking \autoref{cancelacio cp(r) fg} we obtain that $\varphi(x)\leq \varphi(y)$. Hence, applying $\psi$ we see that $x\leq y$.
\end{proof}

Our next goal is to obtain an analogue, for the Cuntz semigroup of a ring with stable rank one, of the cancellation result for the Cuntz semigroup of a $C^*$-algebra with stable rank one obtained by R\o rdam and Winter in \cite[Theorem 4.2]{RorWin10ZRevisited}. The problem may be formulated as follows. For simplicity, we shall only consider unital rings.

\begin{prbl}
\label{prbl:main-cancellation-problem}
Let $R$ be a unital ring of stable rank one, and let $A,B,C,D$ be countably generated right $R$-modules. Suppose that $A\oplus C \precsim B\oplus D,$ and that $D\prec C.$ Is it necessarily true that $A \precsim B$? If not, under what additional assumptions does this hold?
\end{prbl}

We have not been able to answer this question in full generality. In the remainder of this section, we consider several cases in which it admits a positive answer.

\begin{prp} \label{Y direct summand}
Let $R$ be a unital ring of stable rank one, and let $A,B,C,D$ be right $R$-modules. Suppose that $A\oplus C \precsim B\oplus D,$ and that $D\prec_Y C$, where $Y$ is contained in a finitely generated projective direct summand of $C$. Then $A \precsim B$.   
\end{prp} 

\begin{proof}
Let $Y'$ be a finitely generated projective direct summand of $C$ that contains $Y$. Since $D\prec_Y C$ and $Y \subseteq Y'$, it is easy to verify that $D \precsim Y'$. And since $Y'$ is a direct summand of $C$, one also easily obtains that $Y' \precsim C$. Hence, 
\begin{equation*}
    A \oplus Y' \precsim A \oplus C \precsim B \oplus D \precsim B \oplus Y'.
\end{equation*}
Using~\autoref{cancelacio cp(r) fg} yields $A \precsim B$, as desired. 
\end{proof}

As an immediate consequence of the previous proposition and \cite{PunMcGRoth07} and \cite{Warfield1972}, we obtain the following.

\begin{cor}
Let $R$ be a unital ring with stable rank one such that every projective $R$-module is a direct sum of finitely generated modules. Let $A,B,C,D$ be countably generated projective right $R$-modules. Suppose that $A\oplus C \precsim B\oplus D,$ and that $D\prec C$. Then $A \precsim B$. In particular, this is the case for weakly semihereditary rings, exchange rings, and one-sided principal ideal rings. 
\end{cor}

\begin{cor}
Let $R$ be a unital ring of stable rank one, and let $A,B,C,D$ be countably generated projective right $R$-modules. Suppose that $A\oplus C \precsim B\oplus D,$ that $D\precsim C$, and that either $C$ or $D$ is finitely generated. Then $A \precsim B$.
\end{cor}

\begin{proof}
If $C$ is finitely generated then $D \prec_C C$, and $C$ is a finitely generated direct summand of itself. If $D$ is finitely generated, then $D \precsim C$ implies that $D$ is isomorphic to a direct summand $Y$ of $C$ (see \cite[Remark 4.6]{AntAraBosPerVil25}), and hence $D \prec_Y C$. Thus, in either case, the hypothesis of~\autoref{Y direct summand} are satisfied, and the conclusion follows.
\end{proof}

In view of the proof of~\autoref{Y direct summand}, a natural way to relax the assumption that $Y$ is contained in a finitely generated direct summand of $C$ is to require the existence of a module $E$ and a finitely generated projective module $Y'$ such that $Y' \precsim C \oplus E$ and $D \oplus E \precsim Y'$. This motivates the following:

\begin{pgr}[Complements]
\label{dfn:complement}
Let $R$ be a unital ring and let $D, C$ be countably generated right $R$-modules such that $D \subseteq C \subseteq R^n$ for some $n \in \mathbb{N}$. We say that the right $R$-module $E$ is a \emph{complement of} $D \subseteq C$ \emph{in} $R^n$ if it is a countably generated submodule of $R^n$ such that $C + E = R^n$, $D \cap E = 0$ and $D + E$ is a pure submodule of $R^n$. 

In the extreme case that $D = C$, we have that $E$ is a complement of $D \subseteq C$ in $R^n$ if and only if $C \oplus E = R^n$, that is, if $E$ is a complement of $C$ in $R^n$ in the usual sense. Of course, this implies the strong conclusion that $C$ is a finitely generated projective module.

Observe also that any complement $E$ of $D\subseteq C$ in $R^n$ satisfies that $D+E =D\oplus E$ is a pure submodule of $R^n$. Hence, we necessarily have that both $D$ and $E$ are pure in $R^n$, and it follows from~\autoref{rmk:pures-are-projective} that $D$ and $E$ are countably generated projective submodules of $R^n$.

\end{pgr}

\begin{lma}
\label{rmk:remarks-for-complement}
Let $R$ be a unital ring and let $D\subseteq C$ be countably generated right $R$-submodules of $R^n$ for some $n \in \mathbb{N}$. If $E$ is a complement of $D \subseteq C$ in $R^n$, then
\begin{enumerate}[label=\rm{(\roman*)}]
    \item $D \oplus E \precsim R^n \precsim C \oplus E$, and 
    \item $D \prec C$.
\end{enumerate}
\end{lma}
\begin{proof}
(i): As we have observed in \autoref{dfn:complement}, $D + E = D \oplus E$ and, since $D + E$ is pure in $R^n$, \autoref{impl right to left} implies that $D \oplus E \precsim R^n$. In addition, consider the map $\mu \colon C \oplus E \rightarrow R^n$ given by $\mu(c,e) = c + e$. Since $C + E = R^n$, $\mu$ is surjective, and thus it splits as $R^n$ is projective. Therefore, there exists $\sigma \colon R^n \rightarrow C \oplus E$ such that $\mu \circ \sigma = \mathrm{id}_{R^n}$. Consequently, $\sigma(R^n)$ is a direct summand of $C \oplus E$. Since $R^n \cong \sigma(R^n)$ we obtain $R^n \precsim C \oplus E$.

(ii): Retain the notation in the paragraph above, and set $Y \coloneqq \pi_C(\sigma(R^n))$, which is a finitely generated submodule of $C$, where now $\pi_C \colon C \oplus E \rightarrow C$ is the natural projection map. 
We prove $D \prec_Y C$. Let $X$ be a finitely generated submodule of $D$. Then $X + \pi_E(\sigma(R^n))$ is a finitely generated submodule of $D + E$, where $\pi_E \colon C \oplus E \rightarrow E$ is the natural projection map. Since $D + E \leq_{\mathrm{pure}} R^n$, there exists a homomorphism $\rho\colon R^n \rightarrow D + E$ such that
\begin{equation} \label{eq rk rho}
    \rho_{|X + \pi_E(\sigma(R^n))} = \mathrm{id}_{X + \pi_E(\sigma(R^n))}.
\end{equation}
Observe that since $\mu \circ \sigma = \mathrm{id}_{R^n}$, for $x \in X$ we have $\pi_C(\sigma(x)) = x - \pi_E(\sigma(x))$. Now, set 
\begin{equation*}
    \phi \coloneqq \pi_C \circ \sigma_{|X} \colon X \longrightarrow C \qquad \text{ and } \qquad \psi \coloneqq \pi_D \circ \rho_{|C} \colon C \longrightarrow D,
\end{equation*}
where $\pi_D \colon D + E \rightarrow D$ is the natural projection map, which is well-defined because $D \cap E = 0$. Then, if $x \in X$, we have
\begin{equation*}
    \psi\circ\phi (x) = \pi_D(\rho(\pi_C(\sigma(x)))) = \pi_D(\rho(x - \pi_E(\sigma(x)))) = \pi_D(x - \pi_E(\sigma(x))) = x,
\end{equation*}
where the third equality follows from (\ref{eq rk rho}) and the fourth equality holds since $x \in X \subseteq D$, $\pi_E(\sigma(x)) \in E$, and $D \cap E = 0$. Hence, we have a factorization of the inclusion of $X$ in $D$ through $C$. Moreover, $\phi(X) \subseteq \pi_C(\sigma(R^n)) = Y$, as desired.
\end{proof}

\begin{rmk}
It is a well-known and widely used fact that the Cuntz semigroup $\Cu(A)$ of a C*-algebra $A$ satisfies the axiom of almost algebraic order, most commonly referred to as axiom (O5); see \cite[Lemma 7.1]{RorWin10ZRevisited}, and also \cite[Proposition 4.6]{AntPerThi14arX:TensorProdCu}. Under the assumption of stable rank one, this axiom reads as follows: if $x',x,y\in \Cu(A)$ satisfy $x'\ll x\leq y$, then there is $z\in \Cu(A)$ such that $x'+z\leq y\leq x+z$. Our \autoref{rmk:remarks-for-complement} above certifies that having complements in the sense of \autoref{dfn:complement} is an indication that a version of (O5) for the Cuntz semigroup of a ring might be possible.
\end{rmk}

Combining our observations, we obtain the following result.

\begin{thm} \label{complement implica cancelacio}
Let $R$ be a unital ring with stable rank one and let $A,B,C,D$ be countably generated projective right $R$-modules with $D \subseteq C \subseteq R^n$ for some $n \in \mathbb{N}$. Suppose that
\begin{equation*}
    A \oplus C \precsim B \oplus D
\end{equation*}
and that $D \subseteq C$ has a complement in $R^n$. Then $A \precsim B$. 
\end{thm}

\begin{proof}
Let $E$ be a complement of $D \subseteq C$ in $R^n$. By condition (i) in \autoref{rmk:remarks-for-complement}, we get 
\begin{equation*}
A \oplus R^n \precsim A \oplus C \oplus E \precsim B \oplus D \oplus E \precsim B \oplus R^n,
\end{equation*}
so by~\autoref{cancelacio fg} we obtain $A \precsim B$. 
\end{proof}

It is not difficult to produce examples of pure submodules $D\subseteq C\subseteq R^n$, with $D\prec C$, such that the pair $D\subseteq C$ has no complement in $R^n$. We conclude this section by presenting a case in which complements exist. Before doing so, we give a characterization of when the direct sum of two pure submodules of a projective module is pure.

\begin{prp}
    \label{prp:sum-of-pures}
Let $P,Q,M$ be projective modules such that $P$ and $Q$ are pure submodules of $M$. Suppose that 
$P\cap Q= 0$. Then $P+Q$ is pure in $M$ if and only if for each pair of finitely generated submodules $X$ and $Y$ of $P$ and $Q$, respectively, there exists $\phi_1\colon M \to P$ and $\phi_2\colon M\to Q$ such that $\phi_1\phi_2 = 0=\phi_2\phi_1$,
$\phi_1|_X= \mathrm{id}_X$, and $\phi_2|_Y=\mathrm{id}_Y$.
    \end{prp}

\begin{proof}
Suppose that the stated condition is satisfied. Then, given finitely generated modules $X$ and $Y$ of $P$ and $Q$, take 
$\phi:=\phi_1+\phi_2\colon M \to P+Q$, where $\phi_i$ are as in the statement. Then for $y\in Y$
$$\phi_1 (y) = \phi_1(\phi_2 (y))= 0.$$
Similarly $\phi_2 (x) =0$ for all $x\in X$. It follows that $\phi|_{X+Y}=\mathrm{id}_{X+Y}$. This shows that $P+Q$ is pure in $M$.

Conversely, assume that $P+Q$ is pure in $M$. Let $X$ be a finitely generated submodule of $P$ and $Y$ a finitely generated submodule of $Q$. 
Since $P$ is projective, there exists a f.g.
submodule $X'$ of $P$ such that $X\subseteq X'$ and $\rho_1 \in K(P, X')$ such that $\rho_1|_X= \mathrm{id}_X$.  
Similarly, we can find a f.g. submodule $Y'$ of $Q$ such that $Y\subseteq Y'$ and $\rho_2\colon Q \to Y'$ such that
$\rho_2|_Y= \mathrm{id}_Y$. 

Now since $P+Q$ is pure in $M$, there exists $\omega \colon M \to P+Q$ such that $\omega|_{X'+Y'} = \mathrm{id}_{X'+Y'}$.
Let $\omega_1= \pi_P\circ \omega \colon M\to P$, and $\omega_2=\pi_Q\circ \omega\colon M\to Q$. Here $\pi_P$ and $\pi_Q$ are the projections on 
$P$ and $Q$ respectively, induced by the decomposition $P+Q=P\oplus Q$. 

Observe that for $y'\in Y'$ we have
$$\omega_1 (y') =\pi_P (\omega (y')) = \pi_P (y') =0 ,$$
since $y'\in Q$. Similarly $\omega_2(X')=0$. 
Now set $\phi_1:= \rho_1 \circ \omega_1\colon M\to X'\subseteq P$ and $\phi_2:= \rho_2\circ \omega_2 \colon M \to Y'\subseteq Q$.
Observe that 
$$(\omega_1 \circ \rho_2)(Q) \subseteq \omega_1 (Y')=0,\qquad (\omega_2\circ \rho_1)(P) \subseteq \omega_2 (X') = 0,$$
hence $\phi_1\circ \phi_2 = \rho_1\circ (\omega_1\circ \rho_2)\circ \omega_2 =0$ and similarly $\phi_2\circ \phi_1 = 0$.
Moreover, it is clear that $\phi_1|_X= \mathrm{id}_X$ and $\phi_2|_Y= \mathrm{id}_Y$. 

This concludes the proof. 
\end{proof}

\begin{pgr}[Left normal rings]
\label{pgr:normal}
Let $R$ be an arbitrary ring. For given elements $a, b \in R$, we write $a \prec_l b$ provided $ba = a$ and $a \prec_r b$ provided $ab = a$. We say that $R$ is \emph{left normal} if, whenever $a, b, c \in R$ satisfy $a \prec_l b \prec_l c$, then there are $d, e \in R$ such that $a \prec_l d \prec_l e \prec_l c$. We point out that this terminology differs from that of \cite[7.2]{AntAraBosPerVil26}, where $R$ is said to be left normal if $M_\infty(R)$, rather than $R$, is left normal in the sense above. Following \cite{ACPQ}, we reserve the term \emph{stably left normal} for this later property, and use \emph{left normal} for the property defined directly on $R$. Although the notion was first introduced in \cite[7.2]{AntAraBosPerVil26}, we adopt the terminology of \cite{ACPQ}, to which we also refer the reader for a thorough discussion of the notion, its connections to related concepts, and examples of classes of rings satisfying it. There is also a `right version' of the condition, using $\prec_r$ instead of $\prec_l$, which yields the concept of a right normal ring. In case $R$ is left and right normal, then we simply say that it is \emph{normal}. It is easy to see that, if $R$ is a unital ring, the two versions are in fact equivalent, hence a left normal unital ring is always normal. We also define $a \ll b$ provided that $a \prec_l b$ and $a \prec_r b$.
\end{pgr}

\begin{prp} \label{prp existencia complement}
Let $R$ be a unital ring, and let $C$ and $D$ be countably generated projective right $R$-modules with $$D \leq_{\mathrm{pure}} C \subseteq R^n$$ for some $n \in \mathbb{N}$. Suppose that $M_n(R)$ is normal, and suppose there exist homomorphisms $\theta_1, \theta_2, \theta_3 \in K(R^n, C)$ such that $\theta_1 \ll \theta_2 \ll \theta_3$ and ${\theta_1}_{|D} = \mathrm{id}_D$.~Then $D \subseteq C$ has a complement in $R^n$. 
\end{prp}

\begin{proof}
Observe that from $\theta_1 \ll \theta_2 \ll \theta_3$ we obtain $1 - \theta_3 \ll 1 - \theta_2 \ll 1 - \theta_1$, where $1$ means the identity map of $R^n$ and, in particular, $1 - \theta_3 \prec_l 1 - \theta_2 \prec_l 1 - \theta_1$. Set $\sigma_1 \coloneqq 1 - \theta_3$. Since $K(R^n) \cong M_n(R)$ is a left normal ring, there exist homomorphisms $\sigma_2, \sigma_2' \in K(R^n)$ such that 
\begin{equation*}
    \sigma_1 \prec_l \sigma_2 \prec_l \sigma_2' \prec_l 1 - \theta_1.
\end{equation*}
Now, take $\sigma_2 \prec_l \sigma_2' \prec_l 1 - \theta_1$ and repeat the process. This gives us homomorphisms $\sigma_3, \sigma_3' \in K(R^n)$ such that 
\begin{equation*}
    \sigma_2 \prec_l \sigma_3 \prec_l \sigma_3' \prec_l 1 - \theta_1.
\end{equation*}
Repeating inductively this process we get a sequence of homomorphisms $\sigma_i \in K(R^n)$ such that for all $i \geq 2$, 
\begin{equation*}
    1 - \theta_3= \sigma_1 \prec_l \sigma_i \prec_l \sigma_{i+1} \prec_l 1 - \theta_1.
\end{equation*}

Set $E \coloneqq \bigcup_{i \geq 1} \sigma_i(R^n)$ and let us show that $E$ is a complement of $D \subseteq C$ in $R^n$. It is clear that $E \cong \varinjlim (R^n, \sigma_i)$, and a direct limit of this form is known to be a countably generated projective module; see for example~\cite[Theorem 2.1]{Whitehead}, or also~\cite[Lemma 4.7]{AntAraBosPerVil25}. 

Let $e = \sigma_k(x) \in E$ for some $k \in \mathbb{N}$ and $x \in R^n$, and suppose that $e$ also belongs to $D$. Since $\sigma_k \prec_l 1 - \theta_1$, $(1 - \theta_1)\sigma_k = \sigma_k$, and since $e \in D$, 
\begin{equation*}
    e = \sigma_k(x) = (1 - \theta_1)\sigma_k(x) = (1 - \theta_1)(e) = 0.
\end{equation*}
This proves $D \cap E = 0$, as desired. 

To prove that $E + C = R^n$, let $x \in R^n$ and decompose it as $x = (1 - \theta_3)(x) + \theta_3(x)$. The second summand clearly lies in $C$ and since $1 - \theta_3 \prec_l \sigma_i$ for all $i\ge 2$, we get $(1 - \theta_3)(x) \in E$. 

Finally, let us show that $D + E$ is pure in $R^n$. Let $X$ be a finitely generated submodule of $D$ and $Y$ a finitely generated submodule of $E$. Since $D$ is pure $C$, there exists $\psi \in \mathrm{Hom}_R(C, D)$ such that ${\psi}_{|X} = \mathrm{id}_X$. Set $\phi_1 \coloneqq \psi\theta_1 \in K(R^n, D)$ which still satisfies ${\phi_1}_{|X} = \mathrm{id}_X$.  Now, since $Y$ is finitely generated, there is $k \in \mathbb{N}$ such that ${\sigma_k}_{|Y} = \mathrm{id}_Y$. Set $\phi_2 \coloneqq \sigma_k(1 - \theta_1)$. Since ${(1 - \theta_1)}_{|E} = \mathrm{id}_{E}$, we have ${\phi_2}_{|Y} = \mathrm{id}_{Y}$. Moreover, $\psi$ takes values on $D$ and ${\theta_1}_{|D} = \mathrm{id}_D$. Hence, $\phi_2\phi_1 = \sigma_k(1 - \theta_1)\psi\theta_1 = 0$. On the other hand, since $ \theta_1 \sigma_i = 0$ for all $i\in \N$, we get $\phi_1\phi_2 = \psi\theta_1\sigma_k(1 - \theta_1) = 0$. Then, because of~\autoref{prp:sum-of-pures}, we conclude that $D + E$ is pure in $R^n$, as we wanted to prove. 
\end{proof}

\begin{cor}
\label{cor:complement}
Let $R$ be a normal commutative unital ring with stable rank one. Let $A, B \in \mathcal{CP}(R)$ and $I, J$ be countably generated projective modules and pure ideals of $R$ such that $A \oplus I \precsim B \oplus J$ and $J \prec I$. Then $A$ is isomorphic to a pure submodule of $B$.
\end{cor}

\begin{proof}
Since $J \prec I$,~\autoref{main theorem} provides a pure submodule $J' \leq_{\mathrm{pure}} I$ such that $J \cong J'$ and $J'$ is contained in a finitely generated submodule $Y \subseteq I$. Since $I$ is pure in $R$, there exists $\theta_1 \in K(R, I)$ such that $(\theta_1)_{|Y} = \mathrm{id}_Y$. In particular, $(\theta_1)_{|J'} = \mathrm{id}_{J'}$. Now, $\theta_1(R)$ is a cyclic submodule of $I$. Using $I \leq_{\mathrm{pure}} R$ again, there exists $\theta_2 \in K(R, I)$ such that $(\theta_2)_{|\theta_1(R)} = \mathrm{id}_{\theta_1(R)}$. Hence, $\theta_2\theta_1 = \theta_1$, that is, $\theta_1 \prec_l \theta_2$. Similarly, we obtain $\theta_3 \in K(R, I)$ such that $\theta_2 \prec_l \theta_3$. Observe that since $\theta_i \in K(R)\cong R$ and $R$ is commutative, the fact that $\theta_1 \prec_l \theta_2 \prec_l \theta_3$ implies that $\theta_1 \ll \theta_2 \ll \theta_3$. Then, by~\autoref{prp existencia complement}, $J' \subseteq I$ has a complement in $R$, so by~\autoref{complement implica cancelacio}, $A \precsim B$. Now, by~\autoref{main theorem}, $A$ is isomorphic to a pure submodule of $B$, as desired.
\end{proof}

\begin{rmk}
The rings to which the preceding proposition applies include large classes of rings, for example unital commutative exchange rings (\cite{ACPQ} and \cite[Theorem 6]{Yu1995}), or rings of complex-valued continuous functions on a one-dimensional space (\cite[Lemma 7.4]{AntAraBosPerVil26} and \cite{vaserstein1971}). 
\end{rmk}

\end{document}